\documentclass[12pt,a4paper, oneside]{amsart}

\usepackage{amsmath, nicefrac, amsthm, verbatim, amsfonts, mathtools, amssymb, upgreek, xcolor, bbm}
\usepackage{graphics, xspace, enumerate}
\usepackage{stix}
\usepackage{dsfont}
\usepackage[a4paper,margin=2.5cm]{geometry}
\usepackage[bb=boondox]{mathalfa}
\usepackage{centernot}

\usepackage{graphicx}
\usepackage[colorlinks=true,citecolor=red,urlcolor=blue,linkcolor=red,bookmarksopen=true,unicode=true,pdffitwindow=true]{hyperref}
\usepackage[english]{babel}
\usepackage[languagenames,fixlanguage]{babelbib}
\hypersetup{pdfauthor={}}
\hypersetup{pdftitle={}}

\theoremstyle{plain}
\newtheorem*{thm*}{Theorem}
\newtheorem{theorem}{Theorem}
\newtheorem{lemma}{Lemma} % same for example numbers
\newtheorem{proposition}{Proposition}
\newtheorem{corollary}{Corollary}

\theoremstyle{definition}
\newtheorem*{definition*}{Definition}
\newtheorem{definition}{Definition} % definition numbers are dependent on theorem numbers
\newtheorem{remark}{Remark}

\numberwithin{theorem}{section}
\numberwithin{lemma}{section}
\numberwithin{proposition}{section}
\numberwithin{corollary}{section}
\numberwithin{definition}{section}
\numberwithin{remark}{section}
\numberwithin{example}{section}

\newcommand{\HG}[4]{{}_{2}F_{1}\left(#1,#2;#3;#4\right)}
\newcommand{\HGalt}{{}_{2}F_{1}}
\newcommand*{\dif}{\mathop{}\!\mathrm{d}}

\title[Excursion theory for the Wright--Fisher diffusion]{Excursion theory for the Wright--Fisher diffusion}

\author[P.\ A.\ Jenkins]{Paul\ A.\ Jenkins}
\address[Paul\ A.\ Jenkins]{
	Department of Statistics\\University of Warwick\\ United Kingdom\\ and Department of Computer Science\\ University of Warwick\\ United Kingdom}
\email{p.jenkins@warwick.ac.uk}

\author[J.\ Koskela]{Jere\ Koskela}
\address[Jere\ Koskela]{
	School of Mathematics, Statistics and Physics\\ Newcastle University\\ United Kingdom}
\email{jere.koskela@ncl.ac.uk}

\author[V.\ M.\ Rivero]{Victor\ M.\ Rivero}
\address[Victor\ M.\ Rivero]{
    Centro de Investigaci{\'o}n en Matem{\'a}ticas\\ Guanajuato\\ Mexico}
\email{rivero@cimat.mx}

\author[J.\ Sant]{Jaromir\ Sant}
\address[Jaromir\ Sant]{
	Department of Statistics\\University of Oxford\\ United Kingdom}
\email{jaromir.sant@stats.ox.ac.uk}

\author[D.\ Span\`o]{Dario\ Span\`o}
\address[Dario\ Span\`o]{
	Department of Statistics\\University of Warwick\\ United Kingdom}
\email{d.spano@warwick.ac.uk}

\author[I.\ Valenti\'c]{Ivana\ Valenti\'c}
\address[Ivana\  Valenti\'c]{
	Department of Mathematics\\University of Zagreb\\ Croatia}
\email{ivana.valentic@math.hr}


\begin{document}
\allowdisplaybreaks[4]

\begin{abstract}
    In this work, we develop excursion theory for the Wright--Fisher diffusion with mutation. Our construction is intermediate between the classical excursion theory where all excursions begin and end at a single point and the more general approach considering excursions of processes from general sets. Since the Wright--Fisher diffusion has two boundary points, it is natural to construct excursions which start from a specified boundary point, and end at one of two boundary points which determine the next starting point. In order to do this we study the killed Wright--Fisher diffusion, which is sent to a cemetery state whenever it hits either endpoint. We then construct a marked Poisson process of such killed paths which, when concatenated, produce a pathwise construction of the Wright--Fisher diffusion.
\end{abstract}

\maketitle

\section{Introduction}

The Wright--Fisher diffusion is one of the most prominent forwards-in-time models used within mathematical population genetics to describe the way in which population-level allele frequencies change over time. Various genetic phenomena such as mutation and selection can be straightforwardly incorporated, offering a robust and rich probabilistic model to explain the genetic variation present around us. This diffusion has been studied for decades, and many probabilistic aspects are now well understood. Despite its prominent status, both in population genetics and as a canonical diffusion on the unit interval, the Wright--Fisher diffusion falls outside the usual scope of general diffusion theory in several ways: its diffusion coefficient vanishes and fails to be Lipschitz-continuous at the boundary of its domain.

While a whole theory of Wright--Fisher diffusions (as well as other singular diffusions) has been developed to circumvent these issues (see e.g.\ \cite{YamadaWatanabe1971}), there remain questions to which existing tools based on stochastic differential equations (SDEs) and martingale problems are ill-suited. For example, the fact that the diffusion coefficient vanishes at the boundary implies that the laws of any two Wright--Fisher diffusions whose drift coefficients differ at a boundary point are singular, which complicates their use in path-based inference and simulation applications \cite{Griffithsetal2018, JenkinsSpano2017, Santetal2022, Santetal2023}. All currently available descriptions of the Wright--Fisher diffusion are based on either the generator or the corresponding SDE. We provide a third way of viewing the process through the machinery of excursion theory. Aside from its intrinsic interest, the excursion theory viewpoint has a few advantages. By giving a sample-path construction of the process, we obtain a direct description of the lifetime of an allele from the time it first arises in the population. Second, an excursion approach puts the diffusion's boundary behaviour directly under the spotlight and thereby sheds light on how the mutation parameters influence this. Third, we believe that this characterisation of Wright--Fisher paths will also pave the way to produce more powerful statistical methods for inferring the mutation parameters, which to date has received limited attention in the literature due to the aforementioned problems. As an application, we derive the Hausdorff dimension of the time spent by the diffusion at each of the boundary points, and hence show that diffusions with differing mutation parameters are mutually singular, providing an alternative proof of a recent result in \cite{Jenkins24}.

In this article we present an excursion-based construction of the neutral Wright--Fisher diffusion, whose drift coefficient is an affine function of the state. Excursion theory has been a prominent tool in the study of processes with a single boundary point, such as Bessel processes and reflecting Brownian motions \cite{Blumenthal, Ito,PitmanYor81, RogersExcursions, RogersWilliams}. It is a classical area of probability where a diffusion process is viewed not as a single continuous path but rather as a collection of excursions away from a point indexed by the local time at said point. 
Diffusions with a single boundary point are natural subjects for excursion theory because all excursions begin and end at a single point, facilitating a Markovian construction of a path from concatenated excursions. Excursions of processes from more general sets have also been described, but the results are rarely tractable because it is in general not clear how to concatenate the resulting excursions into a path \cite[Section 8]{RogersExcursions}; the theory of \emph{exit systems} is a general version of this idea \cite{Blumenthal, mai:1975}. Our construction is intermediate between these two regimes, featuring excursions which always start from a specified point, and end at one of two points which determine the next starting point.
We are thus able to provide a construction based on a labelled Poisson process of excursions, with the labels determining how excursions connect to form a continuous path.

We start by considering a simple one-dimensional Wright--Fisher diffusion with mutation, for which a number of results are known but scattered across the literature. We bring these results together and present them in Section \ref{SectionElementary} alongside some interesting relations with the hypergeometric $\HGalt$ functions, for which all the necessary known properties are given in Appendix \ref{SectionApendix}.
Closely related to the Wright--Fisher diffusion, and central to the development of an excursion theory, is the corresponding \emph{killed} process, i.e.\ the Wright--Fisher diffusion which is sent to a cemetery state when it hits either of the boundary points 0 or 1.
The killed diffusion coincides with the usual unkilled one in the interior of its domain.
However, killing at the boundary drastically changes various probabilistic objects, such as the resolvent, Laplace transforms of hitting times, and the transition density, all of which we derive and study in Section \ref{SectionKilled}. 
These quantities are of independent interest in applied sciences, e.g.\ neuroscience \cite{d2018jacobi}.
Section \ref{SectionExcursions} then puts together the results obtained in the previous two sections to develop the full excursion theory for the Wright--Fisher diffusion, which allows for an explicit calculation of the Hausdorff dimension of the time spent at the boundary points by said diffusion.

\section{Elementary results for Wright--Fisher diffusions}\label{SectionElementary}

Consider a Wright--Fisher diffusion $X := (X_t)_{t \geq 0}$ with mutation parameters $\theta_1$ and $\theta_2$, that is a diffusion process on $[0,1]$ with generator given by
 \begin{equation}\label{WFGenerator}
        \mathscr{G} = \frac{1}{2}x(1-x)\frac{\dif^{2}}{\dif x^{2}} + \frac{1}{2}\left(\theta_1 - |\theta|x\right)\frac{\dif}{\dif x},
\end{equation}
acting on domain $C^2([0,1])$, where we have introduced the notation $|\theta| := \theta_1+\theta_2$ \cite[Chapter 10]{EthierKurtz}. 
By Feller's boundary classification \cite[Example 15.6.8, p 240]{KarlinTaylor}, both boundary points are regular if
\begin{equation}\label{assumption}
    0<\theta_1, \theta_2<1.
\end{equation}
We will only be interested in this case since regularity of the boundary means that a process started from a boundary point can enter the interior $(0,1)$, and again reach both endpoints.
This will be essential in order to build the excursions started from the boundary points.
However, some results in this section hold for any $\theta_1,\theta_2>0$ and we will therefore emphasise when we do not use the additional assumption.

\begin{remark}
    One could also consider excursions from endpoint $0$ when $0<\theta_1<1$ and $\theta_2\geq 1$, or from endpoint $1$ when $0<\theta_2<1$ and $\theta_1\geq 1$. In those cases, however, the process could never reach the opposite endpoint and therefore we would be back to the standard setting of excursions from a single point.
\end{remark}

The generator can be rewritten as
\begin{equation}\label{generatorWF}
     \mathscr{G} =\frac{1}{m(x)}\frac{\dif}{\dif x}\frac{1}{s'(x)}\frac{\dif}{\dif x},
\end{equation}
where we have that
\begin{equation*}
    m(x)=\frac{\Gamma(|\theta|)}{\Gamma(\theta_1)\Gamma(\theta_2)}x^{\theta_1 -1}(1-x)^{\theta_2 -1}=\frac{x^{\theta_1 -1}(1-x)^{\theta_2 -1}}{B(\theta_1,\theta_2)} 
\end{equation*}
is the density on $(0,1)$ of the speed measure of $X$ with respect to Lebesgue measure, and 
\begin{equation}\label{s'}
    s'(x)=2\frac{\Gamma(\theta_1)\Gamma(\theta_2)}{\Gamma(|\theta|)}x^{-\theta_1}(1-x)^{-\theta_2}= 2B(\theta_1,\theta_2) x^{-\theta_1}(1-x)^{-\theta_2}
\end{equation}
is the derivative of the scale function $s(x)$  (see \cite[Chapter 15]{KarlinTaylor}). We further denote by $m_{a, b}(x) := x^{\theta_1 + a -1}(1-x)^{\theta_2 + b -1}/B(\theta_1 + a,\theta_2 + b)$  the density of a Beta distribution with parameters $\theta_1 + a, \theta_2 + b \in \mathbb{R}_{+}$. Henceforth, we shall reserve the notation $m(x)$ for the density of a $\textnormal{Beta}(\theta_1, \theta_2)$. We note here that the constants in these definitions can be altered by $m(\dif x) \mapsto cm(\dif x)$, $s(\dif x) \mapsto c^{-1}s(\dif x)$, and \eqref{generatorWF} still holds. Some authors use other constants such as $m(x)=2x^{\theta_1 -1}(1-x)^{\theta_2 -1}$. In our normalization $m$ is also the stationary distribution for $X$. We will abuse notation and use $m$ to denote both the speed measure and its density; which one is intended should be clear from the context. With this convention, we also note that $m(\{0\}) =0$, which follows from the relationship \cite[Section I, Number 7, p16]{BorodinSalminen}
    \begin{equation}\label{m0}
    f^{+_s}(0+) = m(\{0\})\mathscr{G}f(0), \qquad f \in C^2([0,1]),
    \end{equation}
    where $f^{+_{s}}(x)$ denotes the right-hand derivative with respect to the scale function, formally
\begin{equation}
    \label{scale-derivative}
    f^{+_{s}}(x) = \lim_{h\searrow 0}\frac{f(x+h)-f(x)}{s(x+h)-s(x)}=\frac{f^+(x)}{s^+(x)},
\end{equation}
and where $f^+$ is a derivative of $f$ from the right; similarly $m(\{1\}) = 0$.

According to formula (7.14) in \cite{crow1956some}, the transition density with respect to the Lebesgue measure for the Wright--Fisher diffusion can, regardless of the assumption $0<\theta_1,\theta_2<1$, be expressed as
\begin{align}\label{transitiondensityWF}
    p(t,x,y) =\frac{\Gamma(\theta_1+\theta_2)}{\Gamma(\theta_1)\Gamma(\theta_2)}y^{\theta_1 -1}(1-y)^{\theta_2 -1}   \sum_{n=0}^{\infty}\frac{1}{\pi_n}e^{-\frac{n(n+|\theta|-1)t}{2}}R_{n}^{(\theta_1, \theta_2)}(x)R_{n}^{(\theta_1, \theta_2)}(y),
\end{align}
where
\begin{equation*}
    R_{n}^{(\theta_1, \theta_2)}(x) = \HG{-n}{n+|\theta|-1}{\theta_2}{1-x}.
\end{equation*}
With $(a)_n:=a(a+1)\ldots(a+n-1)$ defined as the rising factorial (setting $(a)_0 := 1$), the hypergeometric function $\HGalt$ is defined by the power series
\begin{equation*}
    \HG{a}{b}{c}{x}=\sum_{n=0}^{\infty}\frac{(a)_n(b)_n}{(c)_nn!}x^n,
\end{equation*}
under suitable conditions for $a$, $b$ and $c$ (see Appendix \ref{SectionApendix}). 
Since $(-n)_k=0$ for $k\geq n+1$, the function $R_{n}^{(\theta_1, \theta_2)}$ is a polynomial of finite degree.
In fact it is obtained from a Jacobi polynomial by mapping their usual domain $[-1,1]$ to $[0,1]$ through $x \mapsto (x + 1) / 2$ (see \cite[equations (3.7), (3.8)]{GriffithsSpano2010}). The constant $\pi_n$ is
\begin{equation*}
    \pi_n=\int_0^1 \frac{x^{\theta_1-1}(1-x)^{\theta_2-1}}{B(\theta_1,\theta_2)}\left( R_{n}^{(\theta_1, \theta_2)}(x)\right)^2 \dif x=
    \frac{n!(\theta_1)_{(n)}}{(|\theta|+2n-1)(|\theta|)_{(n-1)}(\theta_2)_{(n)}}.
\end{equation*}
We point out that the $\{R_{n}^{(\theta_1, \theta_2)}(x)\}_{n \in \mathbb{N}}$ form an orthogonal polynomial system on the $\textnormal{Beta}(\theta_1, \theta_2)$ distribution, which is the stationary distribution we are interested here.

We denote by $\mathbb{P}_x$ the law of $X$ when the process is started at $x$ and by $\mathbb{E}_{x}[\cdot]$ the corresponding expectation. We denote by $p_{m}(t,x,y)$ the transition density of the process with respect to the speed measure $m$, which is 
\begin{equation*}
    \mathbb{P}_x\left(X_t\in A \right)=\int_{A}p_{m}(t,x,y)\,m(\dif y), 
\end{equation*}
for any $A \in \mathcal{B}([0,1])$ and $x\in [0,1]$. Therefore the transition density (with respect to Lebesgue measure) can be expressed as 
\begin{equation*}
    p(t,x,y)=p_m(t,x,y)m(y)=\frac{y^{\theta_1 -1}(1-y)^{\theta_2 -1}}{B(\theta_1,\theta_2)}p_m(t,x,y).
\end{equation*}
An alternative well-known expression \cite{gri:1979:AAP11:310, gri:li:1983} is given by
\begin{align}\label{TransitionDensityCoalescent}
    p(t, x, y) = \sum_{n = 0}^{\infty}q_{n}^{|\theta|}(t)\sum_{k=0}^{n}\binom{n}{k}x^k(1-x)^{n-k}\frac{y^{\theta_1+k-1}(1-y)^{\theta_2+n-k-1}}{B(\theta_1+k, \theta_2+n-k)},
\end{align}
where $q_{n}^{|\theta|}(t)$ is the transition probability of a pure death process $(D_t)_{t\geq 0}$ on $\mathbb{N}$ having transitions $n \mapsto n-1$ at rate $n(n+|\theta|-1)/2$ and with an entrance boundary at $\infty$; that is, $q_n^{|\theta|}(t) = \lim_{m\to\infty}\mathbb{P}(D_t = n\mid D_0 = m)$. This death process describes the distribution of the number of lineages in a coalescent process when lineages are lost by both mutation and coalescence, and \eqref{TransitionDensityCoalescent} is a manifestation of the famous duality between the coalescent model and the Wright--Fisher diffusion \cite{don:kur:1999:AAP}.  

The probability mass function $\{q_n^{|\theta|}(t):n\in\mathbb{N}\}$ has the form
\begin{align*}
    q^{|\theta|}_n(t)=\sum_{m=n}^\infty e^{-\frac{n(n+|\theta|-1)t}{2}}a_{nm}^{|\theta|}
\end{align*}
with coefficients $a_{nm}^{|\theta|}$ explicit \cite{gri:1979:AAP11:310}. The connection between the expansions in \eqref{transitiondensityWF} and \eqref{TransitionDensityCoalescent} is obtained using the identity
\begin{align*}
\frac{1}{\pi_n}
R_{n}^{(\theta_1, \theta_2)}(x)
R_{n}^{(\theta_1, \theta_2)}(y)
=\sum_{m=0}^na_{n,m}^{|\theta|}\sum_{k=0}^{m}\binom{m}{k}\frac{x^k(1-x)^{m-k}y^{k}(1-y)^{m-k}}{\frac{(\theta_1)_{k}(\theta_2)_{m-k}}{(|\theta|)_{m}}}
    \end{align*}

(see \cite{GriffithsSpano2010, GriffithsSpano2013}), and subsequently inverting the order of summation to get \eqref{TransitionDensityCoalescent}.

We also observe that
\begin{align}\label{TransitionDensityCoalescentm}
    p_{m}(t, x, y) = \sum_{n = 0}^{\infty}q_{n}^{|\theta|}(t)\sum_{k=0}^{n}\binom{n}{k}\frac{x^k(1-x)^{n-k}y^{k}(1-y)^{n-k}}{\frac{(\theta_1)_{k}(\theta_2)_{n-k}}{(|\theta|)_{n}}}.
\end{align}

The Green's function is defined as
\begin{equation}\label{GreenDef}
    G_{\lambda}(x,y) := \int_{0}^{\infty}e^{-\lambda t}p_{m}(t,x,y)\dif t,
\end{equation}
and it is important to note that we define it \emph{with respect to the speed measure $m$}. 

\begin{proposition}
    \label{prop:Green}
Three equivalent formulations for the Green's function $G_{\lambda}(x,y)$ of the Wright--Fisher diffusion are
 \begin{align}
 &G_{\lambda}(x,y) \notag \\
&=\sum_{n=0}^{\infty}\frac{2}{2\lambda+n(n+\theta-1)}\frac{1}{\pi_n}R_{n}^{\theta_1,\theta_2}(x)R_{n}^{\theta_1,\theta_2}(y) \label{eq:greensjacobi}\\
&= \sum_{n=0}^{\infty}\frac{1}{\lambda +\lambda_n}\prod_{j=n+1}^\infty \left(\frac{\lambda_j}{\lambda + \lambda_j}\right)\sum_{k=0}^{n}\binom{n}{k}\frac{x^k(1-x)^{n-k}y^k(1-y)^{n-k}}{\frac{(\theta_1)_{k}(\theta_2)_{n-k}}{(|\theta|)_{n}}}\label{eq:greenproduct}\\
&= \frac{2\Gamma(a(\lambda))\Gamma(b(\lambda))} {\Gamma(|\theta|)}\HG{a(\lambda)}{b(\lambda)}{\theta_1}{x\wedge y}\HG{a(\lambda)}{b(\lambda)}{\theta_2}{1-(x \vee y)},\label{eq:greenwronskian}
    \end{align}
where we define 
\begin{equation}\label{a&b}
    a(\lambda) := \frac{|\theta|-1+\sqrt{(|\theta|-1)^{2}-8\lambda}}{2} \, \, \, \text{ and } \, \, \,  b(\lambda) := \frac{|\theta|-1-\sqrt{(|\theta|-1)^{2}-8\lambda}}{2},
\end{equation}
and set $\lambda_j:=j(j+|\theta|-1)/2$, for any $j\in\mathbb Z_+$.
\end{proposition}
\begin{proof}
Inserting the corresponding terms from expansion \eqref{transitiondensityWF} for the transition density into the definition of $G_\lambda(x,y)$ we, regardless of the assumption $0<\theta_1,\theta_2<1$, get the following:
    \begin{align*}
        G_{\lambda}(x,y) &= \int_{0}^{\infty} e^{-\lambda t}\sum_{n=0}^{\infty}\frac{1}{\pi_n}e^{-\frac{n(n+|\theta|-1)t}{2}}R_{n}^{(\theta_1, \theta_2)}(x)R_{n}^{(\theta_1, \theta_2)}(y) \dif t.
    \end{align*}
Integrating termwise recovers \eqref{eq:greensjacobi}.   

To obtain \eqref{eq:greenproduct}, we use the alternative form \eqref{TransitionDensityCoalescentm} for $p_{m}(t,x,y)$. Using the known identities (e.g.\ \cite[p407]{gri:2006}) 
\begin{align}
    \mathbb{P}\left(\sum_{j=k}^{n}T_j < t\right) &= \mathbb{P}\left(D_t \leq k - 1 | D_{0} = n\right), \label{eq:gri:2006}\\ 
    \int_{0}^{t}q_{m}^{|\theta|}(s)\dif s = \int_0^t \mathbb{P}\left(D_s=m\right) \dif s &=\frac{2}{m(m+|\theta|-1)}\mathbb P\left(\sum_{j=m}^{\infty} T_j<t\right),\label{eq:QtSumExp}
\end{align}
for $m \geq 1$, where $T_j$ are independent exponential random variables with rates $j(j+|\theta|-1)/2$. Equation \eqref{eq:QtSumExp} directly implies for $m \geq 1$ that
\begin{align*}
    \int_{0}^{\infty}e^{-\lambda t}q_{m}^{|\theta|}(t) \dif t = \frac{2}{m(m+|\theta|-1)}\prod_{j=m}^{\infty}\mathbb{E}\left[e^{-\lambda T_j}\right] = \frac{1}{\lambda + \lambda_m}\prod_{j=m+1}^{\infty}\frac{\lambda_j}{\lambda + \lambda_j}.
\end{align*}
For $m = 0$, we observe that taking $n \rightarrow \infty$ in \eqref{eq:gri:2006} implies that $\frac{\dif}{\dif t} q^{|\theta|}_{0}(t)$ corresponds to the density of the random variable $\sum_{j=1}^{\infty}T_j$. Then we use integration by parts to observe that 
\begin{align*}
    \int_{0}^{\infty}e^{-\lambda t}q_{0}^{|\theta|}(t)dt = \frac{1}{\lambda}\left(-\int_{0}^{\infty}\frac{\dif}{\dif t}(e^{-\lambda t}q_{0}^{|\theta|}(t))dt + \int_{0}^{\infty}e^{-\lambda t}\frac{\dif}{\dif t}q_{0}^{|\theta|}(t)dt\right) = \frac{1}{\lambda}\prod_{j=1}^{\infty}\frac{\lambda_j}{\lambda + \lambda_j}.
\end{align*}
Putting all of the above together, we get \eqref{eq:greenproduct}.
To prove \eqref{eq:greenwronskian} we need to introduce a few more concepts, so it is deferred to later in this Section. 
\end{proof}

\begin{remark}
As a byproduct, the following identity, which follows from \eqref{eq:greensjacobi} and \eqref{eq:greenwronskian}, seems to be new and of independent interest for special function theory:
\begin{align*}
&\HG{a(\lambda)}{b(\lambda)}{\theta_1}{x\wedge y}\HG{a(\lambda)}{b(\lambda)}{\theta_2}{1-(x \vee y)}\\
&=\frac{\Gamma(|\theta|)}{2\Gamma(a(\lambda))\Gamma(b(\lambda))}\sum_{n=0}^{\infty}\frac{2}{2\lambda+n(n+\theta-1)}\frac{1}{\pi_n}R_{n}^{\theta_1,\theta_2}(x)R_{n}^{\theta_1,\theta_2}(y).
\end{align*}
\end{remark}

One more ingredient we require to construct and concatenate excursions is the first hitting time, $H_{y} := \inf\left\{ t>0 : X_t = y \right\}$. According to \cite[equation (9)]{PitmanYor2003}, for $\lambda \geq 0$ 
    \begin{align}\label{LaplaceForHittingTime}
            \mathbb{E}_x\left[ e^{-\lambda H_y} \right] &= \begin{cases}
            \frac{\Phi_{\lambda,-}(x)}{\Phi_{\lambda,-}(y)} & x < y, \\ \\
            \frac{\Phi_{\lambda,+}(x)}{\Phi_{\lambda,+}(y)} & x > y,
            \end{cases}
    \end{align}
where the functions $\Phi_{\lambda,\pm}$ are, respectively, the decreasing and increasing solutions to the generator eigenfunction equation
\begin{align}\label{GenEigFns}
    \mathscr{G}\Phi = \lambda \Phi,
\end{align}
subject to suitable boundary conditions which reflect the boundary behaviour of the diffusion. For more details see \cite[Section I, Number 10, p19]{BorodinSalminen}. For a diffusion with speed measure $m$, scale function $s$, and killing measure $k$, the boundary conditions at zero are given by
\begin{align*}
\lambda\Phi_{\lambda,-}(0)m(\{0\}) = \left(\Phi_{\lambda,-}\right)^{+_{s}}(0) - \Phi_{\lambda,-}(0)k(\{0\}) 
\end{align*}
where we recall that, for a function $f$, the derivative $f^{+_s}$ is as defined in \eqref{scale-derivative}.

In the case of a Wright--Fisher diffusion without killing and with $0 < \theta_1, \theta_2 < 1$, we get $k(\{0\}) = 0 = m(\{0\})$, and the boundary condition for \eqref{GenEigFns} becomes $\left(\Phi_{\lambda,-}\right)^{+_{s}}(0) = 0$. For $\Phi_{\lambda,+}(1)$ we have a similar condition, with $\left(\Phi_{\lambda,+}\right)^{-_{s}}$ defined analogously to $\left(\Phi_{\lambda,-}\right)^{+_{s}}$ as a derivative from the left:
\begin{align*}
\lambda\Phi_{\lambda,+}(1)m(\{1\}) = -\left(\Phi_{\lambda,+}\right)^{-_{s}}(1) - \Phi_{\lambda,+}(1)k(\{1\}),
\end{align*}
which again simplifies to $\left(\Phi_{\lambda,+}\right)^{-_{s}}(1) = 0$ in the Wright--Fisher case without killing and with $0 < \theta_1, \theta_2 < 1$.
Notice that the condition for $\Phi_{\lambda,-}$  depends only on the behaviour around $0$, and the condition for $\Phi_{\lambda,+}$ depends only on the behaviour around $1$. 
The killing term plays a crucial role in excursion theory, and a good reference for associated boundary conditions is \cite[p 15--17]{BorodinSalminen}.

Solving the eigenfunction equation \eqref{GenEigFns} for the Wright--Fisher generator \eqref{generatorWF} is equivalent to solving the hypergeometric differential equation for a particular parameter configuration. 
For all $\lambda\neq (|\theta|-1)^2 / 8$, the hypergeometric differential equation admits the hypergeometric function as a solution, and thus provides us with the required function $\Phi_\lambda$. 
By matching the coefficients in \eqref{GenEigFns} with the generators given by \eqref{WFGenerator} and \eqref{HypergeometricDifferentialEquation}, we find that
    \begin{align}
        \Phi_{\lambda,+}(x) &= \HG{\frac{|\theta|-1+\sqrt{(|\theta|-1)^{2}-8\lambda}}{2}}{\frac{|\theta|-1-\sqrt{(|\theta|-1)^{2}-8\lambda}}{2}}{\theta_2}{1-x}\label{Philambda+},\\
        \Phi_{\lambda,-}(x) &= \HG{\frac{|\theta|-1+\sqrt{(|\theta|-1)^{2}-8\lambda}}{2}}{\frac{|\theta|-1-\sqrt{(|\theta|-1)^{2}-8\lambda}}{2}}{\theta_1}{x},\label{Philambda-}
    \end{align}
are, respectively, the decreasing and increasing solutions to the equation \eqref{GenEigFns}.
Examples of $\Phi_{\lambda, -}$ are plotted in Figure \ref{fig:phi-}.

Observe that using \eqref{Philambda+} and \eqref{Philambda-} within \eqref{LaplaceForHittingTime}, it follows that
    \begin{align}
        \mathbb{E}_x\left[ e^{-\lambda H_y} \right] &=\begin{cases}
        \displaystyle\frac{\HG{a(\lambda)}{b(\lambda)}{\theta_1}{x}}{\HG{a(\lambda)}{b(\lambda)}{\theta_1}{y}} & x < y, \\ \\
        \displaystyle\frac{\HG{a(\lambda)}{b(\lambda)}{\theta_2}{1-x}}{\HG{a(\lambda)}{b(\lambda)}{\theta_2}{1-y}}  & x > y.
        \end{cases}\label{LaplaceForHittingTimeSimple}
    \end{align}

\begin{remark}\label{countableset}
    We emphasise that the eigenfunction solutions \eqref{Philambda+}--\eqref{Philambda-} hold for almost all $\lambda$. Contrast this property with the set of eigenfunction solutions to \eqref{GenEigFns} when there is a boundary condition for $\Phi$ at \emph{both} zero and one. In that situation, by Sturm--Liouville theory (e.g.\ \cite[Section 1.4]{DominguezDeLaIglesia} or \cite[Section I.3]{kna:2005}) the spectrum of eigenvalues for which a solution exists is countably infinite, namely $\lambda_n = -n(n+|\theta|-1)/2$, $n=0,1,\dots$. It is this countable collection of eigen-value/-function pairs which is used to construct the spectral expansion for $p(t,x,y)$ appearing in \eqref{transitiondensityWF} (see \cite{schoutens2012stochastic}, \cite[Chapter 15]{KarlinTaylor}, as well as \cite[Chapter 4]{DominguezDeLaIglesia} for a full treatment of spectral representations of the transition densities for diffusions). Also notice that, in obtaining \eqref{Philambda+}--\eqref{Philambda-}, assumption \eqref{assumption} becomes important because if it is not satisfied then functions $\Phi_{\lambda,\pm}$ may not be well defined. For details on this point see Remark \ref{2F1convergence} in Appendix \ref{SectionApendix}.
\end{remark}
\begin{remark}
    For an explanation on why the condition $\lambda\neq (|\theta|-1)^2/8$ is necessary, see Remark \ref{notalllambda} in Appendix \ref{SectionApendix}. This condition will not prove binding because we will mostly be interested in integrals over $\lambda \in [0,\infty)$, and it will turn out that these isolated points will not cause problems.
\end{remark}

Observe that in view of \eqref{hypergeoDerivative} and \eqref{Fin2} we have
\begin{align}
    \frac{ \HG{a(\lambda)}{b(\lambda)}{\theta_1}{\cdot}^+(0) }{s^+(0)}&=\lim_{y\searrow0}\frac{\lambda}{\theta_1}\HG{a(\lambda)+1}{b(\lambda)+1}{\theta_1+1}{y}\frac{y^{\theta_1}(1-y)^{\theta_2}}{2B(\theta_1,\theta_2)} \notag\\
    &=0 \label{condition0}
\end{align}
and
\begin{align*}
    \frac{ \HG{a(\lambda)}{b(\lambda)}{\theta_2}{1-\cdot}^-(1) }{s^-(1)}&=\lim_{y\nearrow1}\frac{-\lambda}{\theta_2}\HG{a(\lambda)+1}{b(\lambda)+1}{\theta_2+1}{1-y}\frac{y^{\theta_1}(1-y)^{\theta_2}}{2B(\theta_1,\theta_2)}\\
    &=0,
\end{align*}
and thus we see that $\Phi_{\lambda,\pm}$ satisfy their respective boundary conditions.

\begin{figure}[ht]
    \centering
    \includegraphics[scale=0.65]{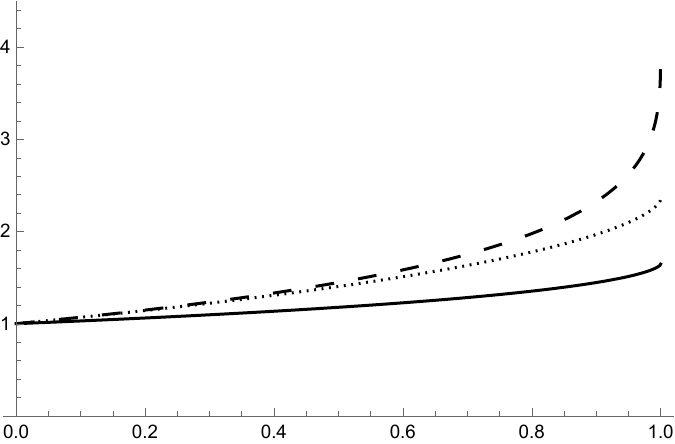}
    \includegraphics[scale=0.65]{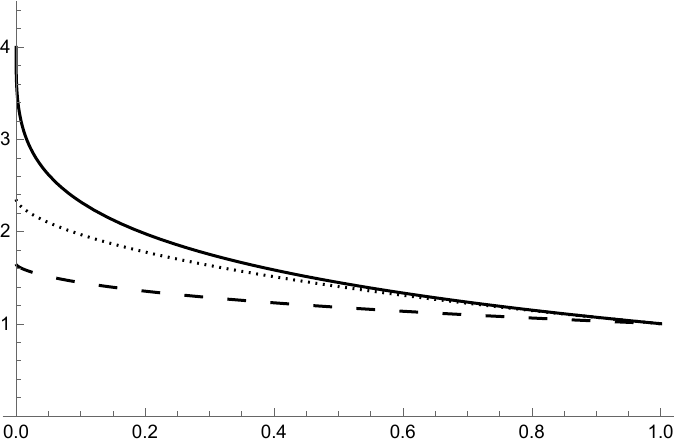}
    \caption{The graphs for $\Phi_{\lambda,-}$ on the left and for $\Phi_{\lambda,+}$ on the right are displayed for $\lambda=0.1$ and $\theta_1=\theta_2=0.3$ dotted, for $\theta_1=0.3, \theta_2=0.7$ dashed and for $\theta_1=0.7, \theta_2=0.3$ bold.}
    \label{fig:phi-}
\end{figure}
The boundary values for $\Phi_{\lambda,\pm}$ will be important in later calculations. Due to \eqref{Fin2} we have that
\begin{equation*}
    \Phi_{\lambda,-}(0)=1=\Phi_{\lambda,+}(1),
\end{equation*}
and from \eqref{Fin1} we have
\begin{equation*}
    \Phi_{\lambda,-}(1)= \frac{\Gamma(\theta_1)\Gamma(1-\theta_2)}{\Gamma\left(\frac{\theta_1-\theta_2+1-\sqrt{(|\theta|-1)^{2}-8\lambda}}{2}\right)\Gamma\left(\frac{\theta_1-\theta_2+1+\sqrt{(|\theta|-1)^{2}-8\lambda}}{2}\right)},
\end{equation*}
and
\begin{equation*}
    \Phi_{\lambda,+}(0)= \frac{\Gamma(\theta_2)\Gamma(1-\theta_1)}{\Gamma\left(\frac{\theta_2-\theta_1+1-\sqrt{(|\theta|-1)^{2}-8\lambda}}{2}\right)\Gamma\left(\frac{\theta_2-\theta_1+1+\sqrt{(|\theta|-1)^{2}-8\lambda}}{2}\right)}.
\end{equation*}

\begin{remark}
 At this juncture we comment on the stature of the results covered thus far. The Wright--Fisher diffusion is by now a classical subject, with several textbook treatments. Particularly useful are \cite{eth:2011}, \cite[Chapter 10]{EthierKurtz}, \cite{ewe:2004:I}, and \cite[Chapter 15]{KarlinTaylor}, which contain most of the details given above. Exceptions are (i) our formulae for the Green's function given in Proposition \ref{prop:Green}, and (ii) the eigenfunctions \eqref{Philambda+}--\eqref{Philambda-} and the Laplace transform \eqref{LaplaceForHittingTimeSimple}. We are not aware of expressions \eqref{eq:greensjacobi}, \eqref{eq:greenproduct}, and \eqref{eq:greenwronskian} appearing elsewhere; they complement expressions for models with an \emph{absorbing} boundary (at least one of $\theta_1 =0$, $\theta_2 = 0$ holds) where the Green's function is well known---see for example \cite[Section 3.5]{eth:2011}, \cite[Section 4.4]{ewe:2004:I}. An early reference is \cite{Ewe:1964}. On (ii), these formulae appear not to be well known in the population genetics literature, though they appear in various forms in work in neuroscience \cite{don:etal:2024, d2018jacobi, lanska1994synaptic} and mathematical finance \cite{campolieti2012}. The usefulness of \eqref{Philambda+}, \eqref{Philambda-}, and \eqref{LaplaceForHittingTimeSimple} in understanding excursions of the Wright--Fisher diffusion therefore seem not to be well explored. From here we believe all formal results to be essentially new, except where stated.
\end{remark}

We are now in a position to prove \eqref{eq:greenwronskian}.

\begin{proof}[Proof of Proposition \ref{prop:Green} representation \eqref{eq:greenwronskian}]
According to \cite[equation (37)]{PitmanYor2003}, we can re-write the Green's function as
\begin{align*}
    G_{\lambda}(x,y) &= \frac{\Phi_{\lambda,-}(x\wedge y)\Phi_{\lambda,+}(x \vee y)}{w_{\lambda}},
\end{align*}
where the \emph{Wronskian} of $\Phi_{\lambda,-}$ and $\Phi_{\lambda,+}$ is defined as
\begin{align}\label{wronskian}
    w_{\lambda} = \frac{\Phi_{\lambda,-}'(y)\Phi_{\lambda,+}(y) - \Phi_{\lambda,-}(y)\Phi_{\lambda,+}'(y)}{s'(y)},
\end{align}
which \emph{does not} depend on $y$. Using the notation $W(f,g)=f(x)g'(x)-f'(x)g(x)$ to capture the numerator of the Wronskian in \eqref{wronskian}, and according to \eqref{W(Phi-,Phi+)} we can conclude that
\begin{align*}
    w_{\lambda} ={}&\frac{\Gamma(|\theta|)}{2\Gamma(\theta_1)\Gamma(\theta_2)} \frac{-W(\Phi_{\lambda,-},\Phi_{\lambda,+})}{y^{-\theta_1}(1-y)^{-\theta_2}} = \frac{\Gamma(|\theta|)}{2\Gamma(\theta_1)\Gamma(\theta_2)}\frac{\Gamma(\theta_1-1)\Gamma(\theta_2)}{\Gamma(a(\lambda))\Gamma(b(\lambda))}\frac{(\theta_1-1)y^{-\theta_1}(1-y)^{-\theta_2}}{y^{-\theta_1}(1-y)^{-\theta_2}} \nonumber\\
    ={}&\frac{\Gamma(|\theta|)}{2\Gamma(a(\lambda))\Gamma(b(\lambda))}.
\end{align*}
Combine this expression for $w_\lambda$ with \eqref{Philambda+} and \eqref{Philambda-} to obtain \eqref{eq:greenwronskian}.
\end{proof}

The resolvent, defined by 
\begin{equation*}
    R_{\lambda}f(x) := \mathbb{E}_{x}\left[ \int_{0}^{\infty}e^{-\lambda t}f(X_t)\dif t \right] 
\end{equation*} 
is given by 
\begin{align}\label{ResolventFormula}
    R_{\lambda}f(x) = \int_{0}^{\infty}\int_{0}^{1}e^{-\lambda t}p(t,x,y)f(y) \dif y\dif t = \int_{0}^{1}G_{\lambda}(x,y)f(y)m(\dif y)
\end{align}
and in the case of Wright--Fisher diffusion we have the following relations.
\begin{proposition}
    Three equivalent formulations for the resolvent of the Wright--Fisher diffusion $R_\lambda$ for a bounded measurable $f$ are
\begin{align}
        &R_\lambda f(x)\nonumber\\ &= \sum_{n=0}^{\infty}\frac{2}{2\lambda+n(n+\theta-1)}\frac{1}{\pi_n}R_{n}^{\theta_1,\theta_2}(x)\mathbb{E}\left[R_{n}^{\theta_1,\theta_2}(Y)f(Y)\right]\label{eq:resolventbinomial} \\
        &=\sum_{n=0}^{\infty}\frac{1}{\lambda + \lambda_n}\prod_{j=n+1}^{\infty}\left(\frac{\lambda_j}{\lambda + \lambda_j}\right)\mathbb{E}[m_{K,n-K}(f)],  \label{eq:resolventbeta}\\
        &=2\frac{\Gamma(a(\lambda))\Gamma(b(\lambda))}{\Gamma(|\theta|)}\mathbb E\left[\HG{a(\lambda)}{b(\lambda)}{\theta_1}{x\wedge Y}\HG{a(\lambda)}{b(\lambda)}{\theta_2}{1-(x \vee Y)}f(Y)\right]\label{eq:resolventwronskian}
    \end{align}
    where $Y$ is a Beta$(\theta_1,\theta_2)$-distributed random variable, $K$ is Binomial$(n,x)$, and $m_{k, n-k}(f) = \int_{0}^{1}f(x)m_{k, n-k}(x)dx$.
\end{proposition}
\begin{proof}
The proposition follows using \eqref{ResolventFormula} in conjunction with Proposition \ref{prop:Green}.
\end{proof}

\section{Killed Wright--Fisher diffusion}\label{SectionKilled}
The next question we are interested in is what happens when we introduce killing at one of the boundary points. For this reason we define $X^y$ as a Wright--Fisher diffusion with generator \eqref{WFGenerator} started at $y\in [0,1]$. For $b\in \{0,1\}$ we will denote by $(X^{y\to b})_{t\geq0}=(X^y_{t \wedge H_b})_{t\geq0}$ the process $X^y$ stopped at the first hitting time $H_{b}$ of point $b$. We will perform calculations for killing at the $0$ endpoint; similar arguments lead to the corresponding results when killing at 1, which we also state here for completeness. When it comes to the transition density of the process $X^{x\to 0}$ killed at $0$, which we denote by $p^{0}(\cdot, \cdot, \cdot)$, a formal calculation provides the following expression:
\begin{align*}
    & p^0(t,x,y)\\
    ={}& p(t,x,y)-\mathbb{E}_x\left[p(t-H_0,X_{H_0},y)1_{\{H_0 <t\}}\right]\\
    ={}&\frac{y^{\theta_1 -1}(1-y)^{\theta_2 -1}}{B(\theta_1,\theta_2)} \sum_{n=0}^{\infty}\frac{1}{\pi_n}e^{-\frac{n(n+|\theta|-1)t}{2}}R_{n}^{(\theta_1, \theta_2)}(x)R_{n}^{(\theta_1, \theta_2)}(y)\\
    &-\frac{y^{\theta_1 -1}(1-y)^{\theta_2 -1}}{B(\theta_1,\theta_2)}\sum_{n=0}^{\infty}\frac{1}{\pi_n}e^{-\frac{n(n+|\theta|-1)t}{2}}R_{n}^{(\theta_1, \theta_2)}(0)R_{n}^{(\theta_1, \theta_2)}(y)\mathbb{E}_x\left[e^{\frac{n(n+|\theta|-1)H_0}{2}}1_{\{H_0 <t\}}\right]\\
    ={}&\frac{y^{\theta_1 -1}(1-y)^{\theta_2 -1}}{B(\theta_1,\theta_2)} \sum_{n=0}^{\infty}\frac{1}{\pi_n}e^{-\frac{n(n+|\theta|-1)t}{2}}R_{n}^{(\theta_1, \theta_2)}(y)\\
    &\times \left[ R_{n}^{(\theta_1, \theta_2)}(x)-\frac{\Gamma(1-\theta_1)\Gamma(\theta_2)}{\Gamma(1-\theta_1-n)\Gamma(\theta_2+n)}\mathbb{E}_x\left[e^{\frac{n(n+|\theta|-1)H_0}{2}}1_{\{H_0 <t\}}\right]\right],
\end{align*}
where we used \eqref{transitiondensityWF} and \eqref{Fin1}.
However, it is difficult to say much more from this without detailed knowledge of the distribution of $H_0$, which is not available beyond its Laplace transform.
An alternative is to follow the spectral approach in \cite[Section 15.13]{KarlinTaylor} to determine \eqref{transitiondensityWF}, with suitable boundary conditions in place to reflect the killing behaviour at $0$, as follows.

Denote by $u(t,x)=\mathbb{E}_x\left[f(X_t)\right]$ and by $u^0(t,x)=\mathbb{E}_x\left[f(X^{x\to 0}_t)\right]$. If we consider the spectral decomposition of the Wright--Fisher diffusion, an expression for $u$ follows from solving the Kolmogorov backward equation
\begin{equation}\label{PDE}
    \frac{\partial u}{\partial t}(t,x) = \mathscr{G}u(t,x),
\end{equation}
where $\mathscr{G}$ is given in \eqref{WFGenerator}. As outlined in Remark \ref{countableset}, there exists a countable set of eigenvalues of the form $\lambda_n = -n(n+|\theta|-1)/2$, $n=0,1,\dots$, and it follows that 
\begin{equation*}
    u(t,x)=\sum_{n=1}^\infty c_n(f) e^{-\frac{n(n+|\theta|-1)}{2}t}R_n^{(\theta_1,\theta_2)}(x),
\end{equation*}
where $R_n^{(\theta_1,\theta_2)}$ are Jacobi polynomials linearly transformed to have domain $[0,1]$. The constants $c_n(f)$ need to be chosen such that the initial condition $u(0,x) = f(x)= \sum_{n=1}^\infty c_n(f) R_n^{(\theta_1,\theta_2)}(x)$ is satisfied. Since linearly transformed Jacobi polynomials are orthogonal with the respect to the scalar product $\langle f,g \rangle=\int_0^1 f(x)g(x)m(x)\dif x$ \cite{GriffithsSpano2010}, it follows that 
\begin{equation*}
    c_n(f)=\frac{\langle f,R_n^{(\theta_1,\theta_2)} \rangle}{\langle R_n^{(\theta_1,\theta_2)},R_n^{(\theta_1,\theta_2)} \rangle}
\end{equation*}
fulfils this requirement.
Choosing a sequence of test functions $f$ approximating a point mass at $y \in [0, 1]$ yields \eqref{transitiondensityWF}.

In the case of the Wright--Fisher diffusion with killing at $0$, the function $u^0$ still needs to satisfy \eqref{PDE} with the same initial condition $u^{0}(0, x) = f(x)$, and also the boundary condition
\begin{equation*}
    u^{0}(t,0)=0.
\end{equation*}
This boundary condition immediately precludes the Jacobi polynomials, and thus we focus instead on Jacobi functions of the second kind, denoted by $Q_n^{\theta_1, \theta_2}(x)$ \cite[Section 4.23]{Szego}. We have
\begin{equation*}
    u^0(t,x)=\sum_{n=1}^\infty c^0_n(f) e^{-\frac{n(n+|\theta|-1)}{2}t}Q_n^{\theta_1, \theta_2}(x),
\end{equation*}
where $Q_n^{\theta_1, \theta_2}$ is not a polynomial but instead 
\begin{equation*}
    Q_n^{\theta_1, \theta_2}(x)=x^{1-\theta_1}\HG{1-n-\theta_1}{n+\theta_2}{2-\theta_1}{x}.
\end{equation*}

Since Jacobi functions of the second kind are not orthogonal with respect to the same scalar product, it is not straightforward to fix constants $c^0_n(f)$ such that the initial condition is satisfied and thus it proves difficult to obtain the transition density of the killed diffusion explicitly. Instead we compute its Laplace transform, for which we can obtain an explicit expression for the appropriate resolvent $R_{\lambda}^{0}$. 

\begin{lemma}\label{lemmaresolventkilled} 
For a bounded measurable $f$ and $x\in [0,1]$, the following holds:
    \begin{align*}
        R_{\lambda}^{0}f(x)= 2\frac{\Gamma(a(\lambda))\Gamma(b(\lambda))}{\Gamma(|\theta|)}\mathbb E\Bigl[\Bigl(&\HG{a(\lambda)}{b(\lambda)}{\theta_1}{x\wedge Y}\HG{a(\lambda)}{b(\lambda)}{\theta_2}{1-(x \vee Y)}\\
        &-\frac{\Gamma(\theta_2-a(\lambda))\Gamma(\theta_2-b(\lambda))}{\Gamma(\theta_2)\Gamma(1-\theta_1)} \HG{a(\lambda)}{b(\lambda)}{\theta_2}{1-x} \\
        &\phantom{-}\times\HG{a(\lambda)}{b(\lambda)}{\theta_2}{1-Y}\Bigr)f(Y)\Big],\\
        R_{\lambda}^{1}f(x)= 2\frac{\Gamma(a(\lambda))\Gamma(b(\lambda))}{\Gamma(|\theta|)}\mathbb E\Bigl[\Bigl(&\HG{a(\lambda)}{b(\lambda)}{\theta_1}{x\wedge Y}\HG{a(\lambda)}{b(\lambda)}{\theta_2}{1-(x \vee Y)}\\
        &-\frac{\Gamma(\theta_1-a(\lambda))\Gamma(\theta_1-b(\lambda))}{\Gamma(\theta_1)\Gamma(1-\theta_2)} \HG{a(\lambda)}{b(\lambda)}{\theta_1}{x} \\
        &\phantom{-}\times\HG{a(\lambda)}{b(\lambda)}{\theta_1}{Y}\Bigr)f(Y)\Big],
    \end{align*}
    where $Y$ is a Beta$(\theta_1,\theta_2)$-distributed random variable.
\end{lemma}
\begin{proof}
    From (1) in \cite{RogersResolvents} we have that
    \begin{equation*}
        R_{\lambda}^{0}f(x)=R_{\lambda}f(x)- \mathbb{E}_x\left[ e^{-\lambda H_0} \right]R_{\lambda}f(0).
    \end{equation*}
    We proceed by using equation \eqref{eq:resolventwronskian} for the resolvent of the process which is not killed, \eqref{LaplaceForHittingTimeSimple} for the calculation of $\mathbb{E}_x\left[ e^{-\lambda H_0} \right]$, where we use the second case since $x>0$, and in the end we use equations \eqref{Fin1} and \eqref{Fin2} for the calculation of the constants. This ends the proof of the first statement, with an analogous proof holding for the second.
\end{proof}

If one was interested in calculating the resolvent of the process killed at time $H_{0,1}=H_0\wedge H_1$, that is the process killed when it reaches either endpoint, then the same formula needs to be applied twice,
\begin{equation*}
    R_{\lambda}^{0,1}f(x)=R^0_{\lambda}f(x)- \mathbb{E}_x\left[ e^{-\lambda H_1^0} \right]R^0_{\lambda}f(1),
\end{equation*}
where $H_1^0$ is the first hitting time of $1$ for the Wright--Fisher diffusion killed at $0$. This implies the following formula
\begin{equation*}
    R_{\lambda}^{0,1}f(x)=R_{\lambda}f(x)- \mathbb{E}_x\left[ e^{-\lambda H_1}1_{H_1<H_0} \right]R_{\lambda}f(1)- \mathbb{E}_x\left[ e^{-\lambda H_0}1_{H_0<H_1} \right]R_{\lambda}f(0)
\end{equation*}
To calculate $\mathbb{E}_x\left[ e^{-\lambda H_1^0} \right]$ we can again use \eqref{LaplaceForHittingTime} but we first need to determine what the corresponding set of increasing and decreasing eigenfunctions $\Phi_{\lambda,-}$ and $\Phi_{\lambda,+}$ are for the killed process. Denote by $\Phi^{b}_{\lambda,\pm}$ the decreasing and increasing solutions to the generator eigenfunction equation 
\begin{equation*}
    \mathscr{G}\Phi = \lambda \Phi,
\end{equation*}
when the boundary conditions are set to immediate killing at the endpoint $b\in\{0,1\}$; that is the killing measure satisfies $k(\{b\}) = \infty$.
\begin{proposition}
    With $a(\lambda)$ and $b(\lambda)$ defined in \eqref{a&b}, the following hold:
    \begin{align}
        &\Phi_{\lambda,+}^{0}(x) = \HG{a(\lambda)}{b(\lambda)}{\theta_2}{1-x},\label{Philambda0+}\\
        &\Phi_{\lambda,-}^{0}(x) = x^{1-\theta_1} \HG{\theta_2-b(\lambda)}{\theta_2-a(\lambda)}{2-\theta_1}{x},\label{Philambda0-}\\
        &\Phi_{\lambda,+}^{1}(x) = (1-x)^{1-\theta_2} \HG{\theta_1-a(\lambda)}{\theta_1-b(\lambda)}{2-\theta_2}{1-x},\label{Philambda1+}\\
        &\Phi_{\lambda,-}^{1}(x) = \HG{a(\lambda)}{b(\lambda)}{\theta_1}{x}.\label{Philambda1-}
    \end{align}
\end{proposition}
\begin{proof}
Since the boundary condition for $\Phi_{\lambda,+}^{0}$ depends only on the behaviour around $b=1$, killing the process at endpoint $b=0$ will not affect it. Therefore the boundary condition remains unchanged, $\left(\Phi_{\lambda,+}^{0}\right)^{-_{s}}(1) = 0$, and we have the formula in \eqref{Philambda0+} the same as \eqref{Philambda+}.

On the other hand, since we introduce killing at the endpoint $b=0$, $k(\{0\}) = \infty$, we have $\Phi_{\lambda,-}^0(0) = 0$ (see \cite[p 15--17]{BorodinSalminen}). The second equation in \eqref{H2F1@0} is a solution to the generator eigenfunction equation, and observe that at $0$
\begin{equation*}
    (0)^{1-\theta_1}\HG{\theta_2-b(\lambda)}{\theta_2-a(\lambda)}{2-\theta_1}{0}=0,
\end{equation*}
which leads to \eqref{Philambda0-}.

When the process is killed upon reaching the endpoint $b=1$, the boundary conditions are
\begin{equation}\label{conditionsb=1}
    \left(\Phi_{\lambda,-}^{1}\right)^{+_{s}}(0) = 0 \text{ and } \Phi_{\lambda,+}^{1}(1) = 0.
\end{equation}
Therefore, clearly by the same arguments as in \eqref{Philambda-} we arrive at \eqref{Philambda1-}. For the second condition in \eqref{conditionsb=1}, we note that the second equation in \eqref{H2F1@1} is a solution which at $1$ gives
\begin{equation*}
    (1-1)^{1-\theta_2}\HG{\theta_1-a(\lambda)}{\theta_1-b(\lambda)}{2-\theta_2}{1-1}=0,
\end{equation*}
and thus \eqref{Philambda1+} holds. 
\end{proof}
\begin{remark}\label{Phiboundary}
    The boundary values for $\Phi^{b}_{\lambda,\pm}$ will again prove important in calculations to follow. Equation \eqref{Fin1} implies that 
        \begin{align}\label{phi+00}
            \Phi_{\lambda,+}^{0}(0) &= \HG{a(\lambda)}{b(\lambda)}{\theta_2}{1}= \frac{\Gamma(\theta_2)\Gamma(1-\theta_1)}{\Gamma(\theta_2-a(\lambda))\Gamma(\theta_2-b(\lambda))},\\
        \label{phi-01}
            \Phi_{\lambda,-}^{0}(1) &= \HG{\theta_2-b(\lambda)}{\theta_2-a(\lambda)}{2-\theta_1}{1}= \frac{\Gamma(2-\theta_1)\Gamma(1-\theta_2)}{\Gamma(1-a(\lambda))\Gamma(1-b(\lambda))},\\
        \label{phi+10}
            \Phi_{\lambda,+}^{1}(0) &= \HG{\theta_1-a(\lambda)}{\theta_1-b(\lambda)}{2-\theta_2}{1}= \frac{\Gamma(2-\theta_2)\Gamma(1-\theta_1)}{\Gamma(1-b(\lambda))\Gamma(1-a(\lambda))},\\
        \label{phi-11}
            \Phi_{\lambda,-}^{1}(1) &= \HG{a(\lambda)}{b(\lambda)}{\theta_1}{1}= \frac{\Gamma(\theta_1)\Gamma(1-\theta_2)}{\Gamma(\theta_1-a(\lambda))\Gamma(\theta_1-b(\lambda))}.
        \end{align}
\end{remark}

\begin{theorem}\label{resolventtwicekilled}    For a bounded measurable $f$, parameter $\lambda \in [0,\infty)\setminus \{(|\theta|-1)^2 / 8\}$ and $x\in [0,1]$, the following holds:
    \begin{align*}
        &R_{\lambda}^{0,1}f(x)= 2\frac{\Gamma(a(\lambda))\Gamma(b(\lambda))}{\Gamma(|\theta|)}\mathbb E\Bigl[\Bigl(\HG{a(\lambda)}{b(\lambda)}{\theta_1}{x\wedge Y}\HG{a(\lambda)}{b(\lambda)}{\theta_2}{1-(x \vee Y)}\\
        &\phantom{=}-\frac{\Gamma(\theta_2-a(\lambda))\Gamma(\theta_2-b(\lambda))}{\Gamma(\theta_2)\Gamma(1-\theta_1)} \HG{a(\lambda)}{b(\lambda)}{\theta_2}{1-x}\HG{a(\lambda)}{b(\lambda)}{\theta_2}{1-Y}\\
        &\phantom{=}-\frac{\Gamma(1-a(\lambda))\Gamma(1-b(\lambda))}{\Gamma(2-\theta_1)\Gamma(1-\theta_2)} x^{1-\theta_1}\HG{\theta_2-b(\lambda)}{\theta_2-a(\lambda)}{2-\theta_1}{x}\\
        &\phantom{=-}\times \left( \HG{a(\lambda)}{b(\lambda)}{\theta_1}{ Y}-\frac{\Gamma(\theta_2-a(\lambda))\Gamma(\theta_2-b(\lambda))}{\Gamma(\theta_2)\Gamma(1-\theta_1)} \HG{a(\lambda)}{b(\lambda)}{\theta_2}{1-Y} \right)\Bigr)f(Y)\Big],
    \end{align*}
    where $Y$ is a Beta$(\theta_1,\theta_2)$-distributed random variable.
\end{theorem}

\begin{proof}
    Recall that \begin{equation*}
    R_{\lambda}^{0,1}f(x)=R^0_{\lambda}f(x)- \mathbb{E}_x\left[ e^{-\lambda H_1^0} \right]R^0_{\lambda}f(1).
\end{equation*}
    From Lemma \ref{lemmaresolventkilled} we get
    \begin{align*}
        R_{\lambda}^{0}f(1)= 2\frac{\Gamma(a(\lambda))\Gamma(b(\lambda))}{\Gamma(|\theta|)}\mathbb E\Bigl[&\Bigl(\HG{a(\lambda)}{b(\lambda)}{\theta_1}{ Y}\\
        &-\frac{\Gamma(\theta_2-a(\lambda))\Gamma(\theta_2-b(\lambda))}{\Gamma(\theta_2)\Gamma(1-\theta_1)} \HG{a(\lambda)}{b(\lambda)}{\theta_2}{1-Y}\Bigr)f(Y)\Big], 
    \end{align*}
    and observe that since $H_1^0$ is the first hitting time of $1$ for the Wright--Fisher diffusion killed at $0$, the first case in \eqref{LaplaceForHittingTime} gives
    \begin{align*}
        \mathbb{E}_x\left[ e^{-\lambda H_1^0} \right]= {}&\frac{\Phi_{\lambda,-}^0(x)}{\Phi_{\lambda,-}^0(1)}=\frac{x^{1-\theta_1}\HG{\theta_2-b(\lambda)}{\theta_2-a(\lambda)}{2-\theta_1}{x}}{\HG{\theta_2-b(\lambda)}{\theta_2-a(\lambda)}{2-\theta_1}{1}}\nonumber\\
        ={}&\frac{\Gamma(1-a(\lambda))\Gamma(1-b(\lambda))}{\Gamma(2-\theta_1)\Gamma(1-\theta_2)} x^{1-\theta_1}\HG{\theta_2-b(\lambda)}{\theta_2-a(\lambda)}{2-\theta_1}{x},
    \end{align*}
    where we used formula \eqref{Philambda0-} for $\Phi_{\lambda,-}^{0}$ and \eqref{Fin2} for the constant. Putting all of the above together gives the result.
\end{proof}

\begin{remark}\label{rem3.5}
    Note that we could have reversed the order of killing, by first killing the process at $1$ and subsequently at $0$, and we would end up with the same formula. In this case we have
    \begin{align*}
        R_{\lambda}^{0,1}f(x)=R_{\lambda}^{1,0}f(x)=R^1_{\lambda}f(x)- \mathbb{E}_x\left[ e^{-\lambda H_0^1} \right]R^1_{\lambda}f(0),
    \end{align*}
    where $H_0^1$ is the first hitting time of $0$ for the Wright--Fisher diffusion killed at $1$. Lemma \ref{lemmaresolventkilled} gives an expression for $R^1_{\lambda}f$ whilst \eqref{LaplaceForHittingTime} leads to
    \begin{align*}
        \mathbb{E}_x\left[ e^{-\lambda H_0^1} \right]&= \frac{\Phi_{\lambda,+}^1(x)}{\Phi_{\lambda,+}^1(0)}\\
        &= \frac{\Gamma(1-a(\lambda))\Gamma(1-b(\lambda))}{\Gamma(2-\theta_2)\Gamma(1-\theta_1)} (1-x)^{1-\theta_2} \HG{\theta_1-a(\lambda)}{\theta_1-b(\lambda)}{2-\theta_2}{1-x},
    \end{align*}
    which implies that
    \begin{align}\label{R10}
        &R_{\lambda}^{0,1}f(x)= 2\frac{\Gamma(a(\lambda))\Gamma(b(\lambda))}{\Gamma(|\theta|)}\mathbb E\Bigl[\Bigl(\HG{a(\lambda)}{b(\lambda)}{\theta_1}{x\wedge Y}\HG{a(\lambda)}{b(\lambda)}{\theta_2}{1-(x \vee Y)} \notag\\
        &\phantom{=}- \frac{\Gamma(\theta_1-a(\lambda))\Gamma(\theta_1-b(\lambda))}{\Gamma(\theta_1)\Gamma(1-\theta_2)} \HG{a(\lambda)}{b(\lambda)}{\theta_1}{x}\HG{a(\lambda)}{b(\lambda)}{\theta_1}{Y} \notag \\
        &\phantom{=}- \frac{\Gamma(1-a(\lambda))\Gamma(1-b(\lambda))}{\Gamma(2-\theta_2)\Gamma(1-\theta_1)} (1-x)^{1-\theta_2}\HG{\theta_1-b(\lambda)}{\theta_1-a(\lambda)}{2-\theta_2}{1-x} \notag\\
        &\phantom{= -}\times \left(\HG{a(\lambda)}{b(\lambda)}{\theta_2}{1-Y} -\frac{\Gamma(\theta_1-a(\lambda))\Gamma(\theta_1-b(\lambda))}{\Gamma(\theta_1)\Gamma(1-\theta_2)}\HG{a(\lambda)}{b(\lambda)}{\theta_1}{Y} \right)\Bigr)f(Y)\Big].
    \end{align}
    Comparing the two formulas for $R_{\lambda}^{0,1}$, it is clear that \eqref{R10} follows from Theorem \ref{resolventtwicekilled} by swapping $0$ with $1$ (and therefore $Y$ with $1-Y$), and $\theta_1$ with $\theta_2$. This is equivalent to the observation that for a Wright--Fisher diffusion $\Tilde{X}$ with mutation coefficients $(\theta_2, \theta_1)$ (i.e.\ the mutation coefficients of $X$ interchanged), the distribution of $1-\Tilde{X}$ is equal to that of $X$.
\end{remark}

\subsection{Estimates of the killed resolvent} In this subsection we derive some estimates for the resolvent of the process killed at its first hitting time of $\{0,1\}$ as the starting point tends to $0$ or $1$, respectively. These are an interesting consequence of the identities obtained above and will be useful in understanding the excursion measure presented in the next section. 

\begin{proposition}\label{killed-est}
For any continuous and bounded test function $f$, we have the following limits
\begin{equation}\label{RK1lim}
\begin{split}
\lim_{x\to 0+}\frac{R^{0,1}_{\lambda} f(x)}{\mathbb{P}_x\left(H_1<H_0\right)}&=\left[\frac{\Gamma(1-a(\lambda))\Gamma(1-b(\lambda))}{\Gamma(2-|\theta|)}\frac{\Gamma(\theta_1)}{\Gamma(\theta_1-b(\lambda))\Gamma(b(\lambda))}\right]\\
&\ \times\left[R_{\lambda}f(0)\frac{\Gamma(1-\theta_2)\Gamma(b(\lambda))}{\Gamma(1-\theta_2+b(\lambda))}-R_\lambda f(1)\frac{\Gamma(\theta_1-b(\lambda))\Gamma(b(\lambda))}{\Gamma(\theta_1)}\right];
\end{split}
\end{equation}
and
\begin{equation}\label{RK2lim}
\begin{split}
\lim_{x\to 1-}\frac{R^{0,1}_{\lambda} f(x)}{\mathbb{P}_x\left(H_0<H_1\right)}&=\left[\frac{\Gamma(1-a(\lambda))\Gamma(1-b(\lambda))}{\Gamma(2-|\theta|)}\frac{\Gamma(\theta_2)}{\Gamma(\theta_2-b(\lambda))\Gamma(b(\lambda))}\right]\\
&\ \times\left[R_{\lambda}f(1)\frac{\Gamma(1-\theta_1)\Gamma(b(\lambda))}{\Gamma(1-\theta_1+b(\lambda))}-R_\lambda f(0)\frac{\Gamma(\theta_2-b(\lambda))\Gamma(b(\lambda))}{\Gamma(\theta_2)}\right].
\end{split}
\end{equation}
\end{proposition}
\begin{proof} We will only give the argument for the proof of the limit in~\eqref{RK1lim}. We deduce \eqref{RK2lim} using the fact, pointed out above, that for a Wright--Fisher diffusion $\Tilde{X}$ with mutation coefficients $(\theta_2, \theta_1)$  the distribution of $1-\Tilde{X}$ is equal to that of $X$.

We start by suitably representing $\mathbb{P}_x\left(H_1<H_0\right)$. The classical solution to the two-sided exit problem for diffusions, see e.g.\ Proposition 3.2 of Chapter VII in \cite{RevuzYor}, states that 
\begin{align*}
    \mathbb{P}_x\left(H_0<H_1\right)=\frac{s(1)-s(x)}{s(1)-s(0)}=1-\mathbb{P}_x\left(H_1<H_0\right),\qquad x\in[0,1].
\end{align*} Moreover, from \eqref{s'} we know that the scale function can be written as 
\begin{align*}
    \frac{s(x)}{2B(\theta_1, \theta_2)}=\int^{x}_{0} y^{-\theta_1}(1-y)^{-\theta_2}\dif y =: B(1-\theta_1, 1-\theta_2; x),\qquad x\in[0,1];
\end{align*}
where $B(1-\theta_1, 1-\theta_2; \cdot)$ denotes the classical incomplete Beta function. The latter can be further written in terms of ${}_{2}F_1$ hypergeometric functions using the identity (8.17.7) in \cite{NIST:DLMF}. This implies in particular that 
\begin{equation}\label{2-sided}
\begin{split}
&\mathbb{P}_x\left(H_1<H_0\right)=\frac{s(x)-s(0)}{s(1)-s(0)}=\frac{B(1-\theta_1, 1-\theta_2; x)}{B(1-\theta_1, 1-\theta_2)}\\&=\frac{1}{B(1-\theta_1, 1-\theta_2)(1-\theta_1)}x^{1-\theta_1}{}_{2}F_1(1-\theta_1, \theta_2; 2-\theta_1; x),\qquad x\in[0,1].
\end{split}
\end{equation}
We recall from Theorem~\ref{resolventtwicekilled} and Remark~\ref{rem3.5} that the Laplace transform of the first exit time from $(0,1)$ from above or below is respectively given by \begin{equation}\label{2-sided1}
\begin{split}
&\mathbb{E}_x\left[e^{-\lambda H_0}1_{\{H_0<H_1\}}\right]=\mathbb{E}_x\left[e^{-\lambda H^{1}_0}\right]\\
&=\frac{\Gamma(1-a(\lambda))\Gamma(1-b(\lambda))}{\Gamma(2-\theta_2)\Gamma(1-\theta_1)}(1-x)^{1-\theta_2}{}_{2}F_1(\theta_1-a(\lambda), \theta_1-b(\lambda); 2-\theta_2; 1-x);
\end{split}
\end{equation}
and
\begin{equation}\label{2-sided2}
\begin{split}
&\mathbb{E}_x\left[e^{-\lambda H_1}1_{\{H_1<H_0\}}\right]=\mathbb{E}_x\left[e^{-\lambda H^{0}_1}\right]\\
&=\frac{\Gamma(1-a(\lambda))\Gamma(1-b(\lambda))}{\Gamma(2-\theta_1)\Gamma(1-\theta_2)}x^{1-\theta_1}{}_{2}F_1(\theta_2-b(\lambda), \theta_2-a(\lambda); 2-\theta_1; x);
\end{split}
\end{equation}
for any $x\in[0,1]$. To prove the claim in the Proposition, we express the resolvent of the process killed at either boundary  in terms of the resolvent of the process un-killed as follows:
\begin{equation*}
\begin{split}
R^{0,1}_{\lambda} f(x)= {} &\left(R_{\lambda}f(x)-R_{\lambda}f(0)\right)+R_{\lambda}f(0)\mathbb{P}_{x}(H_1<H_0)\\
&{}-R_{\lambda}f(1)\mathbb{E}_{x}\left[e^{-\lambda H_1}1_{\{H_1<H_0\}}\right]\\
&{}+R_{\lambda}f(0)\mathbb{E}_{x}\left[\left(1-e^{-\lambda H_0}\right)1_{\{H_0<H_1\}}\right]\\
=: {}& A_x+B_x-C_x+D_x, \qquad \ x\in[0,1],
\end{split}
\end{equation*}
where the above equation holds true for any measurable and bounded function $f$. We will analyse each of these four terms in turn. To analyse the behaviour of $A_x$ as $x\to 0$,  we observe that if $f$ is continuous and supported by $[0,1]$, the resolvent $R_{\lambda}f$  is in the domain of $\mathscr{G}$. By \eqref{scale-derivative} and \eqref{m0} we have 
\begin{equation*}
    \lim_{x\to0+}\frac{A_x}{\mathbb{P}_{x}(H_1<H_0)}=\lim_{x\to 0+}\frac{s(x)}{\mathbb{P}_{x}(H_1<H_0)}\times\lim_{x\to 0+}\frac{A_x}{s(x)}=0. 
\end{equation*}
Moreover, we have
\begin{align*}
    \frac{B_x}{\mathbb{P}_x(H_1<H_0)}=R_{\lambda}f(0).
\end{align*} 
Also, from \eqref{2-sided}, \eqref{2-sided1} and \eqref{2-sided2}, an easy computation shows that  
\begin{equation*}
\begin{split}
\frac{C_x}{\mathbb{P}_x(H_1<H_0)}&=\frac{R_{\lambda}f(1)\mathbb{E}_x\left[e^{-\lambda H_1}1_{\{H_1<H_0\}}\right]}{\mathbb{P}_x(H_1<H_0)}\\
&\xrightarrow[x\to 0+]{} R_{\lambda}f(1)\frac{\Gamma(1-a(\lambda))\Gamma(1-b(\lambda))}{\Gamma(2-\theta_1)\Gamma(1-\theta_2)}(1-\theta_1)B(1-\theta_1,1-\theta_2)\\
&= R_{\lambda}f(1)\frac{\Gamma(1-a(\lambda))\Gamma(1-b(\lambda))}{\Gamma(2-|\theta|)}.
\end{split}
\end{equation*} We are just left to determine the behaviour of $D_x$, which we further write 
\begin{equation*}
D_x=R_{\lambda}f(0)\widetilde{D}_x,\qquad \widetilde{D}_x:=\left[\mathbb{P}_x(H_0<H_1)-\mathbb{E}_x\left(e^{-\lambda H_0}1_{\{H_0<H_1\}}\right)\right]. 
\end{equation*}
Using the identities \eqref{2-sided} and \eqref{2-sided2} we get  
\begin{equation*}
\begin{split}
\widetilde{D}_x&=\frac{1}{(1-\theta_2)B(1-\theta_2, 1-\theta_1)}(1-x)^{1-\theta_2}{}_{2}F_1(1-\theta_2, \theta_1; 2-\theta_2; 1-x)\\
&\quad- \frac{\Gamma(1-a(\lambda))\Gamma(1-b(\lambda))}{\Gamma(2-\theta_2)\Gamma(1-\theta_1)}(1-x)^{1-\theta_2}{}_{2}F_1(\theta_1-a(\lambda), \theta_1-b(\lambda); 2-\theta_2; 1-x)\\
&=\frac{(1-x)^{1-\theta_2}\Gamma(2-|\theta|)}{\Gamma(2-\theta_2)\Gamma(1-\theta_1)}\Bigl[{}_{2}F_1(1-\theta_2, \theta_1; 2-\theta_2; 1-x) \\
&\qquad-\frac{\Gamma(1-a(\lambda))\Gamma(1-b(\lambda))}{\Gamma(2-|\theta|)}{}_{2}F_1(1-\theta_2+b(\lambda), \theta_1-b(\lambda); 2-\theta_2; 1-x)\Bigr].
\end{split}
\end{equation*}
To determine the latter difference, we point out that an application of \eqref{Fin1} gives
\begin{equation*}
{}_{2}F_1(1-\theta_2, \theta_1; 2-\theta_2; 1)=\frac{\Gamma(2-\theta_2)\Gamma(1-\theta_1)}{\Gamma(2-|\theta|)},\end{equation*}
and 
\begin{equation*}
\frac{\Gamma(1-a(\lambda))\Gamma(1-b(\lambda))}{\Gamma(2-|\theta|)}{}_{2}F_1(1-\theta_2+b(\lambda), \theta_1-b(\lambda); 2-\theta_2; 1)=\frac{\Gamma(2-\theta_2)\Gamma(1-\theta_1)}{{\Gamma(2-|\theta|)}}.\end{equation*}
It is now useful to define 
\begin{equation*}
g(x):={}_{2}F_1(1-\theta_2, \theta_1; 2-\theta_2; 1-x),
\end{equation*}
and
\begin{equation*}
h(x):=\frac{\Gamma(1-a(\lambda))\Gamma(1-b(\lambda))}{\Gamma(2-|\theta|)}{}_{2}F_1(1-\theta_2+b(\lambda), \theta_1-b(\lambda); 2-\theta_2; 1-x),
\end{equation*} for $x\in[0,1]$. 
These functions are such that 
\begin{align*}
    h(0)=\frac{\Gamma(2-\theta_2)\Gamma(1-\theta_1)}{\Gamma(2-|\theta|)}=g(0),
\end{align*}
and both are continuously differentiable on $(0,1)$. Moreover, their derivative can be written as 
\begin{equation*}
\begin{split}
g^{\prime}(z)-h^{\prime}(z)={} &-\frac{(1-\theta_2)\theta_1}{2-\theta_2}{}_{2}F_1(2-\theta_2, 1+\theta_1; 3-\theta_2; 1-z)\\
&{}+\frac{\Gamma(1-a(\lambda))\Gamma(1-b(\lambda))}{\Gamma(2-|\theta|)}\times\frac{(1-\theta_2+b(\lambda))(\theta_1-b(\lambda))}{2-\theta_2}\\
&\qquad {}\times{}_{2}F_1(2-\theta_2+b(\lambda), 1+\theta_1-b(\lambda); 3-\theta_2; 1-z).
\end{split}
\end{equation*}
Because 
\begin{align*}
    3-\theta_2-( 2-\theta_2+b(\lambda)+1+\theta_1-b(\lambda))=-\theta_1=
 3-\theta_2-( 2-\theta_2+1+\theta_1)<0,
\end{align*} 
the estimate \eqref{limit2f1} implies
\begin{equation*}
 \begin{split}
&\lim_{z\to 0} z^{\theta_1}(g^{\prime}(z)-h^{\prime}(z))\\
&=(1-\theta_2)\left[\frac{\Gamma(1-a(\lambda))\Gamma(1-b(\lambda))}{\Gamma(2-|\theta|)}\frac{\Gamma(1-\theta_2)\Gamma(\theta_1)}{\Gamma(1-\theta_2+b(\lambda))\Gamma(\theta_1-b(\lambda))}-1\right].
 \end{split}
 \end{equation*}
An application of L'H\^opital's rule gives us
\begin{equation*}
\begin{split}
\lim_{x\to0+}\frac{g(x)-h(x)}{x^{1-\theta_1}}&=\frac{1-\theta_2}{1-\theta_1}\left[\frac{\Gamma(1-a(\lambda))\Gamma(1-b(\lambda))}{\Gamma(2-|\theta|)}\frac{\Gamma(1-\theta_2)\Gamma(\theta_1)}{\Gamma(1-\theta_2+b(\lambda))\Gamma(\theta_1-b(\lambda))}-1\right].
\end{split}
\end{equation*}
and thus
\begin{equation*}
\begin{split}
&\lim_{x\to0}\frac{D_x}{\mathbb{P}_x\left(H_1<H_0\right)}\\
&=R_{\lambda}f(0)\left[\frac{\Gamma(1-a(\lambda))\Gamma(1-b(\lambda))}{\Gamma(2-|\theta|)}\frac{\Gamma(1-\theta_2)\Gamma(\theta_1)}{\Gamma(1-\theta_2+b(\lambda))\Gamma(\theta_1-b(\lambda))}-1\right].
\end{split}
\end{equation*}
We can next add the four obtained limits to conclude the proof:
\begin{equation*}
\begin{split}
&\lim_{x\to 0+}\frac{R^{0,1}_{\lambda} f(x)}{\mathbb{P}_x\left(H_1<H_0\right)}\\
&=R_{\lambda}f(0)-R_\lambda f(1)\left[\frac{\Gamma(1-a(\lambda))\Gamma(1-b(\lambda))}{\Gamma(2-|\theta|)}\right]\\
&\quad+R_{\lambda}f(0)\left[\frac{\Gamma(1-a(\lambda))\Gamma(1-b(\lambda))}{\Gamma(2-|\theta|)}\frac{\Gamma(1-\theta_2)\Gamma(\theta_1)}{\Gamma(1-\theta_2+b(\lambda))\Gamma(\theta_1-b(\lambda))}-1\right]\\
&=\left[\frac{\Gamma(1-a(\lambda))\Gamma(1-b(\lambda))}{\Gamma(2-|\theta|)}\frac{\Gamma(\theta_1)}{\Gamma(\theta_1-b(\lambda))\Gamma(b(\lambda))}\right]\\
&\qquad\times\left[R_{\lambda}f(0)\frac{\Gamma(1-\theta_2)\Gamma(b(\lambda))}{\Gamma(1-\theta_2+b(\lambda))}-R_\lambda f(1)\frac{\Gamma(\theta_1-b(\lambda))\Gamma(b(\lambda))}{\Gamma(\theta_1)}\right].
\end{split}
\end{equation*}
\end{proof}
The following Corollary follows from the above proof.
\begin{corollary}\label{cor:pref00}
\begin{equation*}
\begin{split}
&\lim_{x\to0+}\frac{\mathbb{E}_x\left[(1-e^{-\lambda H_0})1_{\{H_0<H_1\}}\right]}{\mathbb{P}_x\left(H_1<H_0\right)}\\
&=\frac{\Gamma(1-a(\lambda))\Gamma(1-b(\lambda))}{\Gamma(2-|\theta|)}\frac{\Gamma(1-\theta_2)\Gamma(\theta_1)}{\Gamma(\theta_1-a(\lambda))\Gamma(\theta_1-b(\lambda))}-1;
\end{split}
\end{equation*}
\begin{equation*}
\begin{split}
&\lim_{x\to1-}\frac{\mathbb{E}_x\left[(1-e^{-\lambda H_1})1_{\{H_1<H_0\}}\right]}{\mathbb{P}_x\left(H_0<H_1\right)}\\
&=\frac{\Gamma(1-a(\lambda))\Gamma(1-b(\lambda))}{\Gamma(2-|\theta|)}\frac{\Gamma(1-\theta_1)\Gamma(\theta_2)}{\Gamma(\theta_2-a(\lambda))\Gamma(\theta_2-b(\lambda))}-1;
\end{split}
\end{equation*}
and
\begin{equation*}
\begin{split}
\lim_{x\to0+}\frac{\mathbb{E}_x\left[e^{-\lambda H_1}1_{\{H_1<H_0\}}\right]}{\mathbb{P}_x(H_1<H_0)}= \frac{\Gamma(1-a(\lambda))\Gamma(1-b(\lambda))}{\Gamma(2-|\theta|)}=\lim_{x\to1-}\frac{\mathbb{E}_x\left[e^{-\lambda H_0}1_{\{H_0<H_1\}}\right]}{\mathbb{P}_x(H_0<H_1)}.
\end{split}
\end{equation*}
\end{corollary}

\section{Excursions}\label{SectionExcursions}

Standard excursion theory is mostly interested in excursions which start and end at the same point \cite{Blumenthal, RogersExcursions}. Alternatively, one might look at excursions from a general set, which is also a well-understood theory albeit rather more complicated as one needs to track the start- and end-points of excursions in the set of interest (see Section 8 in \cite{RogersExcursions}, Chapter VII in \cite{Blumenthal}).
Excursions of a Wright--Fisher diffusion from the boundary points fall into the latter category, but since there are only two points in our set and the process cannot jump from one point to the other, our theory will also resemble the simpler description of excursions from a single point.
For example, a Wright--Fisher diffusion started from boundary point $0$ experiences a number of excursions away from and ending at $0$, until it enters a formally infinite excursion which is used to indicate that the process has `switched' to the other boundary point. Such excursions are of infinite length because a killing term is introduced at the opposite boundary point which allows such switches to be identified. Once this happens, a new excursion process from $1$ is started, and again the diffusion can be viewed as a collection of excursions all of which start and end at $1$, up until the time the diffusion reaches $0$. Once more the excursion process enters an infinite excursion, indicating that the diffusion was killed upon hitting the $0$ boundary point. Thus, provided $0<\theta_1, \theta_2<1$ (which ensures that the hitting times of both boundary points are almost surely finite), the sample path of a Wright--Fisher diffusion can be recovered by concatenating all its excursions, only this time we must excise from each sample path any time spent at the artificially introduced cemetery state. This is formally done in equation \eqref{eq:main-decomposition}.

We will construct the excursions of $X$ from the set $\{0,1\}$ of boundary points by taking $y\in \left(0,1\right)$, looking at the excursion away from $y$ and then letting $y\to b$, for $b\in \{0,1\}$. 

\begin{definition}
    For $y\in \left[0,1\right]$ we define by $U_y$ to be the set of all continuous paths which start at $y$, move away from it and subsequently stay at $y$ once they hit it, that is
\begin{equation}\label{DefnOfU}
    U_y := \left\{ f \in C(\mathbb{R}_{+},[0,1]) : f^{-1}([0,1]\setminus\{y\}) = \left(0,\eta\right) \textnormal{ for some } \eta\in\left(0,\infty\right] \right\},
\end{equation}
where $C(\mathbb{R}_{+},[0,1])$ is the set of all continuous functions mapping $\mathbb{R}_{+} := [0,\infty)$ into $[0,1]$, and $\eta$ can be viewed as the lifetime of the excursion.
\end{definition}

Excursions will take values in either $U_0$ or $U_1$, switching from $U_0$ to $U_1$ (and vice-versa) at the corresponding diffusion's stopping times $H_1$ (and $H_0$). As detailed above, the continuous diffusion path can be recovered from the resulting ordered list of excursions, indexed by the local time of the process at the boundary points, together with the stopping times. 

The local time of $X^y$ at $y$ is a non-decreasing, continuous process, which only increases on the set $\{t : X^y_t = y\}$.
Therefore an excursion away from $y$ started at local time $\ell$ comes before another one started at local time $\ell'$ if $\ell<\ell'$, thereby offering a natural way of indexing excursions from $y$ through the Markov local time $(L^y_t)_{t\geq0}$ at $y$. We use the It\^o--McKean normalization
\begin{equation*}
    \int_0^t g(X_u)\dif u=\int g(y) L_t^y m(\dif y)
\end{equation*}
for arbitrary non-negative Borel functions $g$. We also define the right-continuous inverse of the Markov local time process, that is $\tau_\ell^y:=\inf\{ u: L_u^y>\ell\}$. 
\begin{remark}
    Note here that the definition of the local time depends on the normalization we used in the definition of the speed measure.
\end{remark}
\begin{remark}
    One might also consider using the semimartingale local time $(\widetilde{L}^y_t)_{t\geq 0}$; that is, using the Meyer--Tanaka normalization 
    \begin{equation*}
        \int_0^t g(X_u)X_u(1-X_u)\dif u=\int g(y) \widetilde{L}_t^y \dif y
    \end{equation*}
    for arbitrary non-negative Borel functions $g$. This approach would not prove useful in our setting, since the semimartingale local time at the boundary for the Wright--Fisher diffusion is zero, that is $\widetilde{L}^b\equiv 0$ for $b\in\{0,1\}$. Although this is probably well known, we provide a proof below.
    \begin{proposition}
    The semimartingale local time at the boundary for the Wright--Fisher diffusion is identically zero, that is 
    \begin{align*}
        \widetilde{L}^{b}\equiv 0 \textnormal{ for } b \in \{0,1\}.
    \end{align*}
    \end{proposition}
     \begin{proof}
        We prove the result for $b=0$; the case $b=1$ is similar. We follow the argument in \cite[Proposition (1.5), Chapter XI]{RevuzYor}. Observe first that the Wright--Fisher diffusion is a semi\-martingale (in virtue of the corresponding SDE admitting a strong solution \cite{sat:1976:OJM}), and thus by a consequence of Tanaka's formula (see \cite[Theorem (1.7), Chapter VI]{RevuzYor}) we can write
        \begin{align*}
            \widetilde{L}^{0}_{t}(X) &= 2\int_{0}^{t}1_{\{0\}}(X_s)\dif V_s
            = 2\int_0^t 1_{\{0\}}(X_s)\left(-\frac{\theta_2}{2} X_s + \frac{\theta_1}{2}(1-X_s)\right)\dif s \\
            &=  \theta_1 \int_{0}^{t}1_{\{0\}}(X_s)\dif s,
        \end{align*}
        where $\dif V_s = (\theta_1(1-X_s)/2 - \theta_2 X_s/2)\dif s$ is the finite-variation component of $\dif X_s$, and where we also made use of the fact that $\widetilde{L}^{0-}_{t} = 0$ trivially. Using the Meyer--Tanaka normalisation above with $g(y) = 1_{(0,\infty)}(y)/[y(1-y)]$:
        \begin{align*}
            t \geq \int_{0}^{t}1_{(0,\infty)}(X_s) \dif s = \int_{0}^{t} g(X_s) X_s(1-X_s) \dif s = \int_{0}^{1}\frac{1}{y(1-y)}\widetilde{L}^{y}_{t}(X)\dif y.
        \end{align*}
        Now, noting that $\widetilde{L}^{y}_{t}(X)$ is a.s.-c\`adl\`ag in $y$, for the right-hand side to be integrable requires $\widetilde{L}^{0}(X) \equiv 0$.
    \end{proof}
    This proof shows that $\int_{0}^{t}1_{\{0\}}(X_s)\dif s = 0$, and since $X$ is ergodic,
    \begin{align*}
    0 = \lim_{t\to\infty}\frac{1}{t}\int_{0}^{t}1_{\{0\}}(X_s)\dif s = \int 1_{\{0\}}(y)\, m(\dif y) = m(\{0\}),
    \end{align*}
    thereby providing another verification that $m(\{0\}) = 0$.
\end{remark}

We have already stated that the sample path of a Wright--Fisher diffusion can be recovered by concatenating all its excursions. To do this formally we follow \cite[Chapter XII, Section 2]{RevuzYor}. Let $z^b$ denote the function which takes the value $b$ at every time.

\begin{definition}
The \emph{excursion process} at $b$ with respect to $X$ is the process $\xi^b = (\xi^b_\ell: \ell \geq 0)$ with values in $U_b \cup \{ z^b \}$ such that
\begin{itemize}
    \item if $\tau^b_\ell - \tau^b_{\ell-} > 0$ then $\xi^b_\ell = (\xi^b_\ell(t): t \geq 0)$ is given by 
    \begin{align*}
    \xi^b_\ell(t) = 1_{[0,\tau^b_\ell - \tau^b_{\ell-}]}(t)X_{\tau^b_{\ell-} + t};
    \end{align*}
    \item if $\tau^b_\ell - \tau^b_{\ell-} = 0$ then $\xi_\ell = z^b$.
\end{itemize}
\end{definition}
The decomposition of $X$ into its excursions at $0$ and at $1$ is therefore given by
\begin{equation} \label{eq:main-decomposition}
X_t = \sum_{\ell \leq L_t^0} \xi^0_\ell(t - \tau^0_{\ell-})1_{[0,H_1(\xi^0_\ell))}(t - \tau^0_{\ell-}) + \left[1 - \sum_{\ell \leq L_t^1} (1-\xi^1_\ell(t - \tau^1_{\ell-}))1_{[0,H_0(\xi^1_\ell))}(t - \tau^1_{\ell-})\right],
\end{equation}
where $H_1(\xi^0_\ell) = \inf\{ t > 0: \xi^0_\ell(t) = 1\}$ and $H_0(\xi^1_\ell) = \inf\{ t > 0: \xi^1_\ell(t) = 0\}$. The indicators $1_{[0,H_{1-b}(\xi^b_\ell))}$ in this decomposition achieve our goal of discarding the time spent by any excursion at its cemetery state (having hit the opposing endpoint).

Although \eqref{eq:main-decomposition} is a compact form, note that at any time $t$ there is at most one `active' summand across the two sums; another way of expressing \eqref{eq:main-decomposition} is to say
\begin{equation}\label{eq:main-decomposition2}
X_t = \xi^b_\ell(t-\tau_{\ell-}^b) \qquad \text{if }t \in [\tau^b_{\ell-},\tau^b_{\ell}],\,b\in\{0,1\},\, \ell \in [0,\infty).
\end{equation}
Because the sample path $X$ is continuous, we will not encounter a situation where the two excursion processes interfere with each other; that is, there is no time for which condition \eqref{eq:main-decomposition2} holds for \emph{both} $b=0$ and $b=1$.

The non-trivial excursions of $X$ can be recorded as a point process $\Xi^{y\to b}$ on $\mathbb{R}_+\times U_y$, such that $(\ell,\xi) \in \Xi^{y\to b}$ if at local time $\ell$ the process $X^{y\to b}$ makes an excursion $\xi$ away from $y$. The main reason that phrasing the path of $X$ in terms of its excursions is particularly insightful is the following result of It{\^o}.
\begin{theorem}[It{\^o} \cite{Ito}]
For $y \in [0,1]$ and $b\in \{0,1\}$ the excursion process $\Xi^{y\to b}$ is a Poisson point process with intensity measure $\textnormal{Leb}\otimes n^{y\to b}$, where $n^{y\to b}$ is a $\sigma$-finite measure on $U_y$ called the excursion (or characteristic) measure. Thus
\begin{enumerate}
    \item For any $A\subset U_y$ for which $n^{y\to b}(A) < \infty$, we have that 
    \begin{equation}\label{Poisson1}
        \mathbb{P}( \textnormal{no points of } \Xi^{y\to b} \textnormal{ fall in } (0,\ell)\times A ) = e^{-\ell n^{y\to b}(A)}.
    \end{equation}
    \item For disjoint $A_1,\dots,A_k \subset U_y$ with $n^{y\to b}(A_i) < \infty$ $\forall i$, if $A = \cup_{i=1}^{k} A_i$ and $T := \inf\{ \ell > 0 : \Xi^{y\to b}_\ell \in A \}$, then
    \begin{equation*}
        \mathbb{P}( \Xi^{y\to b}_T \in A_i ) = \frac{n^{y\to b}(A_i)}{n^{y\to b}(A)}.
    \end{equation*}
\end{enumerate}
\end{theorem}
\begin{remark}
Due to \eqref{Poisson1}, the distribution of the local time at which the first excursion of process $X^{y\to b}$ in the set $A$ occurs is $\textnormal{Exp}(n^{y\to b}(A))$.
\end{remark}
In view of \eqref{DefnOfU}, excursions are segments of the sample path which avoid the point $y$, and thus we can describe their behaviour by considering the transition semigroup of the process which is killed upon hitting $y$. Here we need to be somewhat careful with the notation because we are already considering the excursions of a killed process $X^{y\to b}$.
\begin{definition}
        For $b\in \{0,1\}$ and $A \in \mathcal{B}( [0, 1])$ by $_{y}P^{\to b}_{t}(x,A)$ we denote
\begin{align*}
    _{y}P^{\to b}_{t}(x,A) := \mathbb{E}_{x}\left[1_A(X^{x\to b}_t) 1_{\{t < H_y\}}\right]= \mathbb{E}_{x}\left[1_A(X_t) 1_{\{t < H_y\wedge H_b\}}\right]
\end{align*}
and the law of $(X_{t \wedge H_y\wedge H_b})_{t\geq0}$ by $_{y}\mathbb{P}^{x\to b}$ when the process is started at $x$.
\end{definition}
The behaviour of the process once it is ``inside'' the excursion, i.e.\ away from $y$ is described by $_{y}{P}^{\to b}_{t}$. We shall be interested in the cases when either $y=0$ and $b=1$, or $y=1$ and $b=0$, for which the resolvent applied to indicator functions has already been determined in \eqref{resolventtwicekilled}.
It remains to specify how the excursion ``enters'' $[0,1]\setminus \{y\}$, which is captured by the following definition of the \emph{entrance law} $n^{y\to b}_t(\cdot)$ associated with $_y P^{\to b}_{t}$. 
\begin{definition}
    For $t > 0$, $A \in \mathcal{B}( [0, 1] \setminus \{y\})$
    \begin{equation*}
        n^{y\to b}_{t}(A) := n^{y\to b}\left(\left\{ f \in U_y : f(t) \in A,\, t < \eta \right\}\right).
    \end{equation*}
    where we recall that $\eta$ denotes the lifetime of the excursion $f$.
\end{definition}
Then the following theorem (see Theorem 5 in \cite{RogersExcursions} or equation $(48.3)$ in \cite{RogersWilliams}) confirms the Markov property for the excursion process and explains how these excursions behave.
\begin{theorem}
Given $0 < t_1 < \dots < t_k$ and $A_1,\dots,A_k \in \mathcal{B}( [0, 1] \setminus \{y\})$,
\begin{align*}
    n^{y\to b}&\left(\left\{ f \in U_y : f_{t_i} \in A_i, \forall i=1,\dots,k,\, t_k < \eta \right\}\right) \\
    &\qquad{}= \int_{A_{1}} \int_{A_2} \dots \int_{A_k} n^{y\to b}_{t_1}(\dif  x_1)_{y}P^{\to b}_{t_2-t_1}(x_1,\dif x_2) \dots _{y}P^{\to b}_{t_k-t_{k-1}}(x_{k-1},\dif x_{k}).
\end{align*}
\end{theorem}
Here, the excursion is initiated according to the entrance law $n^{y\to b}_{t}$ and propagates inside $[0, 1] \setminus \{y\}$ according to $_y P^{\to b}_{\cdot}$ until the process hits $y$.
If we can find both $n^{y\to b}_{t}$ and $_y P^{\to b}_{t}$ then we can fully characterise the excursion process. 

We will in fact determine the Laplace transform of the entrance law,
\begin{equation*}
    n^{y\to b}_{\lambda}(\dif x) := \int_{0}^{\infty}e^{-\lambda t}n^{y\to b}_{t}(\dif x)\dif t,
\end{equation*}
where, in a slight abuse of notation, $n^{y\to b}_\lambda$ will always be reserved for the Laplace transform of the entrance law $n^{y\to b}_t$. The first step towards this is the following lemma.

\begin{lemma}\label{lemmatotalmass}
    For all $\lambda \in [0,\infty)\setminus \{(|\theta|-1)^2 / 8\}$, the following hold:
    \begin{align*} 
        \lambda n^{0\to 1}_\lambda 1= {}&\frac{\Gamma(|\theta|)\Gamma(1-a(\lambda))\Gamma(1-b(\lambda))}{2\Gamma(1-\theta_1)\Gamma(\theta_2)\Gamma(\theta_1-a(\lambda))\Gamma(\theta_1-b(\lambda))},\\
        \lambda n^{1\to 0}_\lambda 1= {}&\frac{\Gamma(|\theta|)\Gamma(1-a(\lambda))\Gamma(1-b(\lambda))}{2\Gamma(\theta_1)\Gamma(1-\theta_2)\Gamma(\theta_2-a(\lambda))\Gamma(\theta_2-b(\lambda))}.
    \end{align*}
\end{lemma}
\begin{proof}
    By noting our choice of normalisation for the local time and using equations (16), (17) and (22) (together with the subsequent discussion) in \cite{PitmanYor2003} (see also \cite[Section 6.1]{ItoMcKean} and \cite[Remark 9.8(ii)]{PitmanYor81}), we get that
    \begin{equation}\label{ItoMcKean}
        \lambda n^{y\to b}_{\lambda}1 = \frac{(\Phi^{b}_{\lambda,-})'(y)}{s'(y)}\frac{1}{\Phi^{b}_{\lambda,-}(y)}-\frac{(\Phi^{b}_{\lambda,+})'(y)}{s'(y)}\frac{1}{\Phi^{b}_{\lambda,+}(y)}
    \end{equation}
    for $b\in \{0,1\}$. By \eqref{ItoMcKean},
    \begin{align*}
         \lambda n^{y\to 1}_\lambda 1 &= \frac{(\Phi^{1}_{\lambda,-})'(y)}{s'(y)}\frac{1}{\Phi^{1}_{\lambda,-}(y)}-\frac{(\Phi^{1}_{\lambda,+})'(y)}{s'(y)}\frac{1}{\Phi^{1}_{\lambda,+}(y)} \\
         &= \frac{1}{s'(y)} \frac{(\Phi^{1}_{\lambda,-})'(y)\Phi^{1}_{\lambda,+}(y)-(\Phi^{1}_{\lambda,+})'(y)\Phi^{1}_{\lambda,-}(y)}{\Phi^{1}_{\lambda,-}(y)\Phi^{1}_{\lambda,+}(y)} \\
         &= -\frac{\Gamma(\theta_1-1)\Gamma(2-\theta_2)}{\Gamma(\theta_1-a(\lambda))\Gamma(\theta_1-b(\lambda))}\frac{\Gamma(|\theta|)}{2\Gamma(\theta_1)\Gamma(\theta_2)}y^{\theta_1}(1-y)^{\theta_2} \frac{(1-\theta_1)y^{-\theta_1}(1-y)^{-\theta_2}}{\Phi^{1}_{\lambda,-}(y)\Phi^{1}_{\lambda,+}(y)}\\
         &=\frac{\Gamma(|\theta|)\Gamma(2-\theta_2)}{2\Gamma(\theta_2)\Gamma(\theta_1-a(\lambda))\Gamma(\theta_1-b(\lambda))} \frac{1}{\Phi^{1}_{\lambda,-}(y)\Phi^{1}_{\lambda,+}(y)}, 
    \end{align*}
    where in going from the second line to the third we used \eqref{W(Phi1-,Phi1+)} and \eqref{s'}.
    Therefore, letting $y \to 0$ and using \eqref{phi+10}, and the facts that $\Phi^{1}_{\lambda,-}(y)$ and $\Phi^{1}_{\lambda,+}(y)$ are bounded continuous functions over $[0,1]$, and that $\Phi_{\lambda,-}^{1}(0) = \HG{a(\lambda)}{b(\lambda)}{\theta_1}{0}=1$, we get
    \begin{align*}
        \lambda n^{0\to 1}_\lambda 1 = \lim_{y\rightarrow 0}\lambda n_{\lambda}^{y\to 1}1 = \frac{\Gamma(|\theta|)\Gamma(2-\theta_2)}{2\Gamma(\theta_2)\Gamma(\theta_1-a(\lambda))\Gamma(\theta_1-b(\lambda))} \frac{\Gamma(1-b(\lambda))\Gamma(1-a(\lambda))}{\Gamma(2-\theta_2)\Gamma(1-\theta_1)}.
    \end{align*}
    Similarly, we have
    \begin{align*}
         \lambda n^{y\to 0}_\lambda 1 &= \frac{(\Phi^{\to 0}_{\lambda,-})'(y)}{s'(y)}\frac{1}{\Phi^{\to 0}_{\lambda,-}(y)}-\frac{(\Phi^{\to 0}_{\lambda,+})'(y)}{s'(y)}\frac{1}{\Phi^{\to 0}_{\lambda,+}(y)} \notag \\
         &= \frac{1}{s'(y)} \frac{(\Phi^{\to 0}_{\lambda,-})'(y)\Phi^{\to 0}_{\lambda,+}(y)-(\Phi^{\to 0}_{\lambda,+})'(y)\Phi^{\to 0}_{\lambda,-}(y)}{\Phi^{\to 0}_{\lambda,-}(y)\Phi^{\to 0}_{\lambda,+}(y)} \notag\\
         &= -\frac{\Gamma(1-\theta_1)\Gamma(\theta_2)}{\Gamma(a(\lambda)-\theta_1+1)\Gamma(b(\lambda)-\theta_1+1)}\frac{\Gamma(|\theta|)}{2\Gamma(\theta_1)\Gamma(\theta_2)}y^{\theta_1}(1-y)^{\theta_2} \frac{(1-\theta_1)y^{-\theta_1}(1-y)^{-\theta_2}}{\Phi^{\to 0}_{\lambda,-}(y)\Phi^{\to 0}_{\lambda,+}(y)} \notag \\
         &= \frac{\Gamma(|\theta|)\Gamma(2-\theta_1)}{2\Gamma(\theta_1)\Gamma(\theta_2-a(\lambda))\Gamma(\theta_2-b(\lambda))} \frac{1}{\Phi^{\to 0}_{\lambda,-}(y)\Phi^{\to 0}_{\lambda,+}(y)}
    \end{align*}
    and thus it follows that
    \begin{equation*}
        \lambda n^{1\to 0}_\lambda 1 = \lim_{y\rightarrow 1}\lambda n_{\lambda}^{y\to 0}1 \frac{\Gamma(|\theta|)\Gamma(2-\theta_1)}{2\Gamma(\theta_1)\Gamma(\theta_2-a(\lambda))\Gamma(\theta_2-b(\lambda))} \frac{\Gamma(1-a(\lambda))\Gamma(1-b(\lambda))}{\Gamma(2-\theta_1)\Gamma(1-\theta_2)}.
    \end{equation*}
\end{proof}

From \cite[equation (26)]{PitmanYor2003}, we have that
\begin{equation*}
    n^{0\to 1}(\text{excursions with infinite lifetime})= \lim_{\lambda\to 0}\lambda n^{0\to 1}_\lambda 1,
\end{equation*}
which motivates the following corollary giving us the rate of infinite excursions.
\begin{corollary}\label{cor:infinite_lifetime}
     The following hold:
    \begin{align*}
        \lim_{\lambda\to 0}\lambda n^{0\to 1}_\lambda 1={}&\frac{\Gamma(|\theta|)\Gamma(2-|\theta|)}{2\Gamma(\theta_1)\Gamma(\theta_2)\Gamma(1-\theta_1)\Gamma(1-\theta_2)}=\frac{1}{2B(\theta_1,\theta_2)B(1-\theta_1,1-\theta_2)},  \\
        \lim_{\lambda\to 0}\lambda n^{1\to 0}_\lambda 1={}&\frac{\Gamma(|\theta|)\Gamma(2-|\theta|)}{2\Gamma(\theta_1)\Gamma(\theta_2)\Gamma(1-\theta_1)\Gamma(1-\theta_2)}=\frac{1}{2B(\theta_1,\theta_2)B(1-\theta_1,1-\theta_2)}.
    \end{align*}
\end{corollary}
\begin{proof}
    Since
    \begin{equation*}
        \lim_{\lambda\to 0} a(\lambda)=|\theta|-1 \text{ and } \lim_{\lambda\to 0} b(\lambda)=0
    \end{equation*}
    we have
    \begin{align*}
        \lim_{\lambda\to 0} \lambda n^{0\to 1} 1&=\frac{\Gamma(|\theta|)\Gamma(2-|\theta|)\Gamma(1)}{2\Gamma(1-\theta_1)\Gamma(\theta_2)\Gamma(1-\theta_2)\Gamma(\theta_1)},\\
        \lim_{\lambda\to 0} \lambda n^{1\to 0} 1&=\frac{\Gamma(|\theta|)\Gamma(2-|\theta|)\Gamma(1)}{2\Gamma(\theta_1)\Gamma(1-\theta_2)\Gamma(1-\theta_1)\Gamma(\theta_2)}.
    \end{align*}
\end{proof}
Since infinite excursions in our setting imply that killing has occurred, we have established the rate at which the excursions from $U_0$ switch to excursions from $U_1$ and vice versa.

In the following result we obtain quantities similar to those in  Lemma~\ref{lemmatotalmass} but singling  into four excursion paths types: those that start at $0$ and end at $0$, respectively, start at $0$ and end at $1$, start at $1$ and end at $0$, and that start at $1$ and end at $1$. 

\begin{lemma}\label{f00}
For any $\lambda\geq 0$,
\begin{align*}
\varphi_{0,0}(\lambda)&:=n^{0\to 1}\left(\left(1-e^{-\lambda H_0}\right)1_{\{H_0<H_1\}}\right)\\
&=\frac{1}{2}\left[\frac{\Gamma(|\theta|)\Gamma(1-a(\lambda))\Gamma(1-b(\lambda))}{\Gamma(1-\theta_1)\Gamma(\theta_2)\Gamma(\theta_1-a(\lambda))\Gamma(\theta_1-b(\lambda))}-\frac{\Gamma(|\theta|)\Gamma(2-|\theta|)}{\Gamma(\theta_1)\Gamma(\theta_2)\Gamma(1-\theta_1)\Gamma(1-\theta_2)}\right], \\
\varphi_{0,1}(\lambda)&:=n^{0\to 1}\left(\left(1-e^{-\lambda H_0}\right)1_{\{H_1<H_0\}}\right)\\
&=\frac{1}{2}\left[\frac{\Gamma(|\theta|)\Gamma(2-|\theta|)}{\Gamma(\theta_1)\Gamma(\theta_2)\Gamma(1-\theta_1)\Gamma(1-\theta_2)}-\frac{\Gamma(|\theta|)\Gamma(1-a(\lambda))\Gamma(1-b(\lambda))}{\Gamma(\theta_1)\Gamma(\theta_2)\Gamma(1-\theta_1)\Gamma(1-\theta_2)}\right]\\
&=n^{1\to 0}\left(\left(1-e^{-\lambda H_0}\right)1_{\{H_0<H_1\}}\right)=:\varphi_{1,0}(\lambda),
\end{align*}
and,
\begin{equation*}
\begin{split}
\varphi_{1,1}(\lambda)&:=n^{1\to 0}\left(\left(1-e^{-\lambda H_1}\right)1_{\{H_1<H_0\}}\right)\\
&=\frac{1}{2}\left[\frac{\Gamma(|\theta|)\Gamma(1-a(\lambda))\Gamma(1-b(\lambda))}{\Gamma(1-\theta_2)\Gamma(\theta_1)\Gamma(\theta_2-a(\lambda))\Gamma(\theta_2-b(\lambda))}-\frac{\Gamma(|\theta|)\Gamma(2-|\theta|)}{\Gamma(\theta_1)\Gamma(\theta_2)\Gamma(1-\theta_1)\Gamma(1-\theta_2)}\right].
\end{split}
\end{equation*}
\end{lemma}

\begin{remark}
    We point out that the identities obtained in Lemmas~\ref{lemmatotalmass} and \ref{f00} are complementary. Indeed, under $n^{0\to1}$, $H_0$ is considered as infinite  on the event where the excursion hits one before returning to zero, thus we have 
    \begin{equation*}
        \begin{split}
    n^{0\to 1}(1-e^{-\lambda H_0})&=n^{0\to 1}(H_0=\infty)+n^{0\to 1}((1-e^{-\lambda H_0})1_{\{H_0<H_1\}})\\
    &=n^{0\to 1}(H_0=\infty)+\varphi_{0,0}(\lambda)=\lambda n^{0\to1}_{\lambda}1,\qquad\forall \lambda\geq 0.
        \end{split}
    \end{equation*}
\end{remark}

\begin{proof}[Proof of Lemma~\ref{f00}]
We just sketch the proof of the first case. The proof for the three other cases follow exactly the same line. First of all, we have that
\begin{equation*}
    \begin{split}
        n^{0\to 1}\left(\left(1-e^{-\lambda H_0}\right)1_{\{H_0<H_1\}}\right)=\lim_{x\to0+}\frac{\mathbb{E}_x\left[\left(1-e^{-\lambda H_0}\right)1_{\{H_0<H_1\}}\right]}{\mathbb{P}_x(H_1<H_0)}n^{0\to 1}(H_1<H_0)
    \end{split}
\end{equation*}
This limit is a consequence of the Markov property under the excursion measure $n^{0\to 1}$, just as in the proof of Theorem 6.1 in~\cite{Blumenthal} page 176. Indeed, by the monotone convergence theorem and the Markov property, we have
\begin{equation*}
    \begin{split}
       &n^{0\to 1}\left(\left(1-e^{-\lambda H_0}\right)1_{\{H_0<H_1\}}\right)=\lim_{\epsilon \to 0+} n^{0\to 1}\left(\left(e^{-\lambda \epsilon}-e^{-\lambda H_0}\right)1_{\{\epsilon<H_0<H_1\}}\right)\\&=\lim_{\epsilon \to 0+} e^{-\lambda\epsilon}n^{0\to 1}\left(\mathbb{E}_{X_{\epsilon}}\left[\left(1-e^{-\lambda H_0}\right)1_{\{H_0<H_1\}}\right]1_{\{\epsilon<H_0\wedge H_1\}}\right).
    \end{split}
\end{equation*}
Because $n^{0\to 1}$ is carried by the paths that start at zero continuously, under this measure $X_\epsilon\to 0+$ as $\epsilon\to 0+$. Call $$r_\lambda=\lim_{x\to 0+}\frac{\mathbb{E}_{x}\left[\left(1-e^{-\lambda H_0}\right)1_{\{H_0<H_1\}}\right]}{\mathbb{P}_x(H_1<H_0)}.$$ We know from Corollary~\ref{cor:pref00} that this limit is well defined and finite. A further application of the Markov property guarantees that  
\begin{equation*}
    \begin{split}
       &n^{0\to 1}\left(\left(1-e^{-\lambda H_0}\right)1_{\{H_0<H_1\}}\right)=r_{\lambda}\lim_{\epsilon \to 0+} n^{0\to 1}\left(\mathbb{P}_{X_{\epsilon}}(H_1<H_0)1_{\{\epsilon<H_0\wedge H_1\}}\right)\\
       &=r_{\lambda} \lim_{\epsilon \to 0+}n^{0\to 1}\left(1_{\{\epsilon<H_1<H_0\}}\right)\\
       &=r_{\lambda}n^{0\to 1}(H_1<H_0).
    \end{split}
\end{equation*}
We derive the claimed expression by using the identities from Corollary~\ref{cor:pref00} and that 
\begin{equation}\label{pfff}
  \frac{\Gamma(|\theta|)\Gamma(2-|\theta|)}{2\Gamma(\theta_1)\Gamma(\theta_2)\Gamma(1-\theta_1)\Gamma(1-\theta_2)}=\lim_{\lambda\to0}\lambda n^{0\to1}_{\lambda}1=n^{0\to 1}(H_1<H_0),  
\end{equation} as proved in Corollary~\ref{cor:infinite_lifetime}. 
\end{proof}

Now we are ready to derive $n^{0\to 1}_{\lambda}$, which together with Theorem \ref{resolventtwicekilled}, fully characterises the Wright--Fisher diffusion via its excursions.

\begin{theorem}\label{excursion-resolvent}
    For all $\lambda \in [0,\infty)\setminus \{(|\theta|-1)^2 / 8\}$, the following hold:
    \begin{align*}
        &n^{0\to 1}_{\lambda}(\dif x)\\
        &= \frac{\Gamma(a(\lambda))\Gamma(b(\lambda))\Gamma(1-a(\lambda))\Gamma(1-b(\lambda))}{2\Gamma(1-\theta_1)\Gamma(\theta_2)\Gamma(\theta_1-a(\lambda))\Gamma(\theta_1-b(\lambda))}\frac{\Gamma(|\theta|)}{\Gamma(\theta_1)\Gamma(\theta_2)}x^{\theta_1 -1}(1-x)^{\theta_2 -1}\\
        &\phantom{=} \times\left[\HG{a(\lambda)}{b(\lambda)}{\theta_2}{1-x}  -\frac{\Gamma(\theta_1-a(\lambda))\Gamma(\theta_1-b(\lambda))}{\Gamma(\theta_1)\Gamma(1-\theta_2)}  \HG{a(\lambda)}{b(\lambda)}{\theta_1}{x} \right]\dif x,\\
        &n^{1\to 0}_{\lambda}(\dif x) \\
        &=\frac{\Gamma(a(\lambda))\Gamma(b(\lambda))\Gamma(1-a(\lambda))\Gamma(1-b(\lambda))}{2\Gamma(1-\theta_2)\Gamma(\theta_1)\Gamma(\theta_2-a(\lambda))\Gamma(\theta_2-b(\lambda))}\frac{\Gamma(|\theta|)}{\Gamma(\theta_1)\Gamma(\theta_2)}x^{\theta_1 -1}(1-x)^{\theta_2 -1}\\
        &\phantom{=}\times\left[\HG{a(\lambda)}{b(\lambda)}{\theta_1}{x}  -\frac{\Gamma(\theta_2-a(\lambda))\Gamma(\theta_2-b(\lambda))}{\Gamma(\theta_2)\Gamma(1-\theta_1)}  \HG{a(\lambda)}{b(\lambda)}{\theta_2}{1-x} \right]\dif x.
    \end{align*}
\end{theorem}
\begin{proof}
As in \cite[p.\ 321]{RogersExcursions} (or equation $(50.16)$ in \cite{RogersWilliams}), for $y\in\left(0,1\right)$ and $b\in \{0,1\}$, we have
\begin{equation*}
    n^{y\to b}_{\lambda}(A) = \lambda\left(n^{y\to b}_{\lambda}1\right)  R^b_{\lambda}1_{A}(y),
\end{equation*}
where for $b\in \{0,1\}$, we denote by $R_\lambda^{b}$ the resolvent of $X^{y\to b}$, that is, the process killed at $b$. From Lemma \ref{lemmaresolventkilled}, we know the formula for $R^b_{\lambda}1_{A}$. Since all the functions under the integral are bounded on the compact space $[0,1]$, the result will follow from
\begin{equation*}
    n^{0\to b}_{\lambda}(\dif x) =\lim_{y\to 0} n^{y\to b}_{\lambda}(\dif x),
\end{equation*}
using the result from Lemma \ref{lemmatotalmass}.
\end{proof}
    
    \begin{figure}[ht]
        \centering
        \includegraphics[scale=0.6]{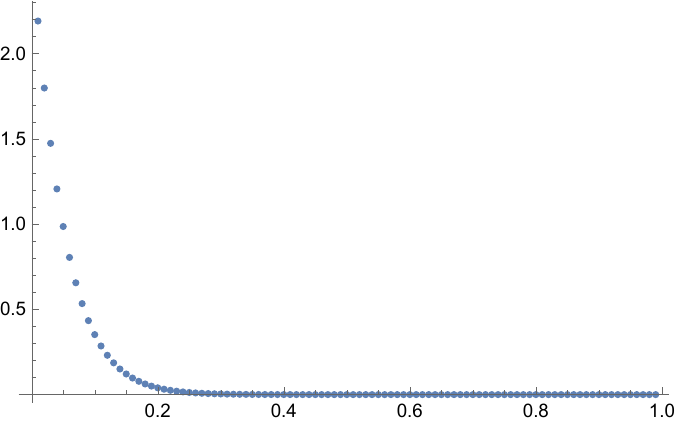}
        \includegraphics[scale=0.6]{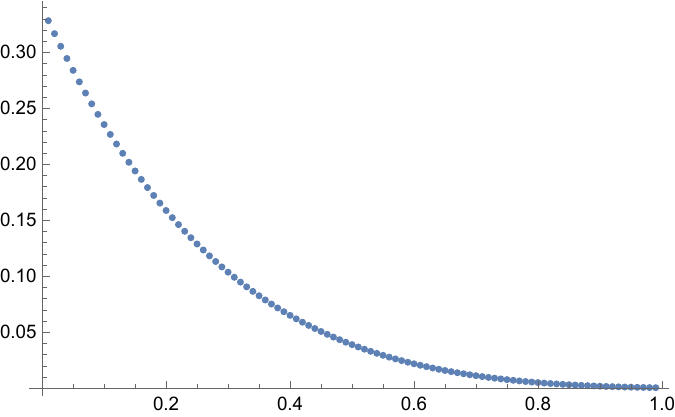}
        \includegraphics[scale=0.6]{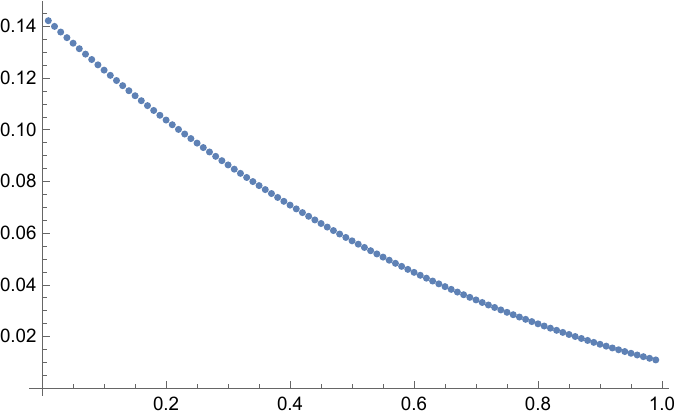}
        \includegraphics[scale=0.6]{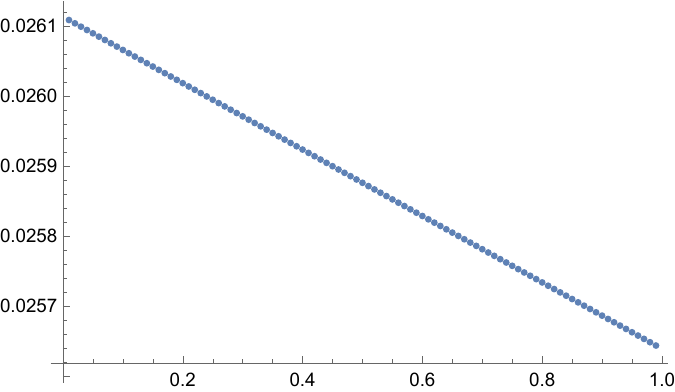}
        \caption{Numerically determined density $n^{0\to 1}_t(\dif x)/\dif x$ for $t=0.1,0.5,1,5$, generated using the \texttt{NInverseLaplaceTransform} function in Mathematica 13.3.
        The notebook used to generate the figures is available as an online supplement.
        }
        \label{fig:numerics}
    \end{figure}

\begin{remark}
    Alternatively, we could have proved Theorem~\ref{excursion-resolvent} by noticing that for any $f$ continuous and bounded function, we have the identity 
    $$\int^{1}_{0}n^{0\to 1}_{\lambda}(\dif y)f(y)=\lim_{x\to 0+}\frac{R^{0,1}_\lambda f(x)}{\mathbb{P}_{x}(H_1<H_0)}\times n^{0\to 1}(H_1<H_0),$$ which is a consequence of the Markov property under the excursion measure $n^{0\to 1}$, as in~Lemma~\ref{f00}. With this at hand, one can then use the identity~\eqref{RK1lim} together with the identities from Corollary~\ref{cor:infinite_lifetime} and \eqref{eq:resolventwronskian}.
\end{remark}
\begin{remark}
    From a theoretical point of view the Laplace transform $n^{0\to 1}_{\lambda}$ characterises the entrance law $n_t$ uniquely. In practice however, inverting the Laplace transform explicitly is only possible in a few select cases and thus is typically done numerically. We illustrate what the resulting entrance law looks like via numerical inversion in Figure \ref{fig:numerics}.
\end{remark}

\subsection{The Hausdorff dimension of the zero set and mutual singularity}
 In this section we prove an interesting consequence of our results. We will show that knowledge of (the leading terms of) the Laplace transform of the excursion measure is sufficient to study absolute continuity (or lack thereof) between any two Wright--Fisher diffusions satisfying \eqref{assumption}, as a function of their respective mutation parameters. We recover some results very recently obtained, by completely different techniques, by \cite{Jenkins24}. Namely, we will show that any two Wright--Fisher diffusions with accessible boundaries differing in their mutation parameters have mutually singular path measures. The proof is here obtained via an analysis of the Hausdorff dimension of the diffusion's zero-  and one- set. Recall that the Hausdorff dimension is an index for the size of a set, in this case the set of times where the allele of type $1$, respectively of type $2$, is not present. We refer to Chapter 5 in~\cite{bertoinsub} where various different notions of fractal dimensions are studied for random regenerative sets, such as that of the times where $X$ is at zero.
\begin{theorem}\label{HD}
    Let $\mathcal{R}_0$ and $\mathcal{R}_1$ be the closure of the set of times where the Wright--Fisher diffusion is at $0$, respectively at $1$, viz. 
    \begin{align*}
        \mathcal{R}_0=\overline{\{t\geq 0: X_t=0\}},\qquad \mathcal{R}_1=\overline{\{t\geq 0: X_t=1\}}.    
    \end{align*}
    We have that for any fixed, arbitrary time $T>0$,  
    \begin{align*}
        \dim_{\text{H}}(\mathcal{R}_0\cap[0,T])=1-\theta_1,\qquad \mathbb{P}_0\text{--a.s.}\qquad\text{and}\qquad  \dim_{\text{H}}(\mathcal{R}_1\cap[0,T])=1-\theta_2,\qquad \mathbb{P}_1\text{--a.s.}
    \end{align*} 
    where $\dim_{\text{H}}$ denotes the Hausdorff dimension. 
\end{theorem}
\begin{corollary}\label{cor:singular}
Denote by $\mathbb{P}_x^{\theta_1,\theta_2,T}$ the path measure up to a fixed, arbitrary time $T > 0$ associated with $X\mid X_0 = x$, $x \in [0,1]$ with the generator $\mathscr{G}$ as in \eqref{WFGenerator}, and suppose $(\theta_1,\theta_2), (\theta_1',\theta_2') \in (0,1)^2$ with $(\theta_1',\theta_2') \neq (\theta_1,\theta_2)$. Then $\mathbb{P}_x^{\theta_1,\theta_2,T}$ and $\mathbb{P}_x^{\theta_1',\theta_2',T}$ are mutually singular.
\end{corollary}
\begin{proof}[Proof of Corollary~\ref{cor:singular}]
It suffices to find an event $A$ such that 
\begin{equation}\label{singular}
\mathbb{P}_x^{\theta_1,\theta_2,T}(A) > 0 = \mathbb{P}_x^{\theta_1',\theta_2',T}(A).
\end{equation}
Without loss of generality suppose $\theta_1 \neq \theta_1'$. Since $X$ is regular, we know $\mathbb{P}_x^{\theta_1,\theta_2,T}(H_0 < T) > 0$ and so, since $\mathcal R_{0}$ is a measurable function of $(X_t)_{t \in [0,T]}$ on $\{H_0 < T\}$, taking $A := \{ \omega \in \Omega:\: \dim_{\text{H}}(\mathcal{R}_0(\omega)\cap[0,T])=1-\theta_1\}$ satisfies \eqref{singular}. Hence $\mathbb{P}_x^{\theta_1,\theta_2,T} \centernot\ll \mathbb{P}_x^{\theta_1',\theta_2',T}$. By a symmetric argument, 
$\mathbb{P}_x^{\theta_1,\theta_2,T} \centernot\ll \mathbb{P}_x^{\theta_1',\theta_2',T}$ also. 
\end{proof}

The proof of the theorem is based on the Lemma~\ref{f00}, where we determined the total mass of the excursion measure on the event that the excursions start and end at the point $0$, respectively at $1$. This quantity determines the so called Laplace exponent of the inverse local time at zero, taking only into account the excursions that start and end at the same endpoint. 
\begin{proof}[Proof of Theorem~\ref{HD}]
    We will just prove the result for $\mathcal{R}_0$; the proof for $\mathcal{R}_1$ follows by interchanging the parameters $\theta_1$ and $\theta_2$ and interchanging the states $0$ and $1$. The proof in this case follows from an application of Corollary~5.3 in~\cite{bertoinsub}. To that end, we should determine the so-called lower index of the Laplace exponent of the inverse local time at $0$, for the Wright--Fisher diffusion issued from $0$. Actually, we notice that only the excursions that start and end at zero are relevant. This is a consequence of the fact that the intensity of the excursions that start at $0$ and end at $1$ is finite as it was proved in Corollary~\ref{cor:infinite_lifetime}. Then, in the local time scale, these `switching' kind of excursions appear only after a strictly positive amount of time, distributed as an exponential r.v.\ with parameter given by~\eqref{pfff}. Hence, in the time scale of the Wright--Fisher diffusion, the times where the process starts an excursion from $0$ and ends it at $1$ are isolated points in the set $\mathcal{R}_0$, and hence they do not contribute to the Hausdorff dimension of the time spent at the state $0$, where there is no presence of the alternative allele. 

     We will hence next determine the lower index of $\varphi_{0,0}$, 
     \begin{equation*}
         \begin{split}
    \underline{\text{ind}}(\varphi_{0,0})&:=\sup\{\rho>0: \lim_{\lambda\to\infty}\varphi_{0,0}(\lambda)\lambda^{-\rho}=\infty\}\\
             &=\sup\left\{\rho>0:\frac{\Gamma(1-a(\lambda))\Gamma(1-b(\lambda))}{\lambda^{\rho}\Gamma(1-\theta_2+b(\lambda))\Gamma(\theta_1-b(\lambda))}=\infty\right\},
         \end{split}
     \end{equation*}
where in the second equality we already used the identity obtained in Lemma~\ref{f00}. Next, using the definition of $a(\lambda)$ and $b(\lambda)$ in~\eqref{a&b},
it is worth noting that for $\lambda$ large enough 
\begin{equation*}
\begin{split}
\frac{3}{2}-\frac{\theta_1+\theta_2}{2}-i\sqrt{2\lambda-\left(\frac{|\theta|-1}{2}\right)^{2}}&=1-a(\lambda)=\overline{1-b(\lambda)}\\
&=\overline{\frac{3}{2}-\frac{\theta_1+\theta_2}{2}+i\sqrt{2\lambda-\left(\frac{|\theta|-1}{2}\right)^{2}}},
\end{split}
\end{equation*}
and 
\begin{equation*}
\begin{split}
\frac{1}{2}+\frac{\theta_1-\theta_2}{2}-i\sqrt{2\lambda-\left(\frac{|\theta|-1}{2}\right)^{2}}&=
1-\theta_2+b(\lambda)=\overline{\theta_1-b(\lambda)}\\
&=\overline{\frac{1}{2}+\frac{\theta_1-\theta_2}{2}+i\sqrt{2\lambda-\left(\frac{|\theta|-1}{2}\right)^{2}}};
\end{split}
\end{equation*}
where as usual $\overline{z}$ denotes the complex conjugate of any complex number $z$.
For notational convenience we denote 
\begin{align*}
    \delta_1=\frac{3}{2}-\frac{\theta_1+\theta_2}{2}, \qquad \delta_2=\frac{1}{2}+\frac{\theta_1-\theta_2}{2},
\end{align*} 
and
\begin{align*}
    z_\lambda:=\sqrt{2\lambda-\left(\frac{|\theta|-1}{2}\right)^{2}},
\end{align*} 
which is real for all $\lambda\geq (|\theta|-1)^{2}/8$. Thus, we are left to determine the more tractable expression
\begin{equation*}
         \begin{split}
             \underline{\text{ind}}(\varphi_{0,0})=\sup\left\{\rho>0:\lim_{\lambda\to\infty}\frac{\Gamma(\delta_{1}+iz_{\lambda})\Gamma(\delta_{1}-iz_{\lambda})}{\lambda^{\rho}\Gamma(\delta_{2}+iz_{\lambda})\Gamma(\delta_{2}-iz_{\lambda})}=\infty\right\}.
         \end{split}
     \end{equation*}
     We can next apply the estimate in (5.11.12) in \cite{NIST:DLMF} to get that 
     \begin{align*}         \frac{\Gamma(\delta_{1}+iz_{\lambda})\Gamma(\delta_{1}-iz_{\lambda})}{\Gamma(\delta_{2}+iz_{\lambda})\Gamma(\delta_{2}-iz_{\lambda})}\sim (z_{\lambda})^{2(\delta_1-\delta_2)},\qquad\lambda\to\infty.
     \end{align*}
     From where it follows that 
     \begin{align*}
         \underline{\text{ind}}(\varphi_{0,0})=\delta_{1}-\delta_2=1-\theta_1,
     \end{align*} 
     which concludes the proof.
\end{proof}

\appendix
\renewcommand{\theequation}{\thesection\arabic{equation}}
\setcounter{equation}{0}

\section{Hypergeometric functions}\label{SectionApendix}

This appendix reviews facts about hypergeometric functions which are relevant for this paper.
The main reference from which they are drawn is \cite{AbramowitzStegun}.
In general, all the parameters $a$, $b$, and $c$, as well as the variable $x$ of the hypergeometric function can be complex numbers. In our case the parameters will sometimes be complex, but the variable will always be a real number.
\subsection{Definition, radius of convergence and boundary values}
The rising factorial is defined by $(a)_n:=a(a+1)\ldots(a+n-1)$ and we have the following trivial relation to the Gamma function: $(a)_n\Gamma(a)=\Gamma(a+n)$.
For $\Re(c-a-b) > -1$, the Gauss hypergeometric function is defined by the power series
\begin{equation}\label{H2F1}
    \HG{a}{b}{c}{x}:=\sum_{n=0}^{\infty}\frac{(a)_n(b)_n}{(c)_nn!}x^n,
\end{equation}
which converges absolutely on $|x| \leq 1$ when $\Re(c-a-b)>0$, and converges conditionally on $|x| < 1$ when $-1<\Re(c-a-b)\leq 0$.
The series is not well defined when $c=-m$ for $m=0,1,2\ldots$ provided $a$ or $b$ is not a negative integer $-n$ with $n<m$. When $a$ or $b$ is equal to $-n$, for $n=0,1,2,\ldots$ the series reduces to a polynomial of degree $n$.
\begin{remark}\label{2F1convergence}
    In all instances of importance in this paper, the parameter $c$ will be one of $\{\theta_1, \theta_2, 2-\theta_1, 2-\theta_2\}$.
    We will also have that $c-a-b> -1$, mostly $c-a-b=1-\theta_2>0$ or $c-a-b=1-\theta_1>0$.
    Hence, by \eqref{assumption}, the power series \eqref{H2F1} will not be divergent.
\end{remark}
When $c\neq 0,-1, -2, \ldots$ and $\Re(c-a-b)>0$, we have the following important identity:
\begin{equation}\label{Fin1}
    \HG{a}{b}{c}{1} = \frac{\Gamma(c)\Gamma(c-a-b)}{\Gamma(c-a)\Gamma(c-b)}.
\end{equation}
Furthermore, when instead $a,b,c$ are such that $\Re(c-a-b)<0$ one has that 
\begin{equation}\label{limit2f1}
\lim_{z\to 1^{-}}(1-z)^{b+a-c}\HG{a}{b}{c}{z}=\frac{\Gamma(c)\Gamma(a+b-c)}{\Gamma(a)\Gamma(b)};
\end{equation} See e.g.\ (15.4.23) in \cite{NIST:DLMF}. Further identities to deal with the case where $\Re(c-a-b)=0$, are known but we will not need them here, see (15.4.21-22) op.cit.
\begin{remark}
    We use this formula repeatedly throughout the paper: for example in Remark \ref{Phiboundary} where, to get \eqref{phi-11}, one needs to check that $0<\theta_1-a(\lambda)-b(\lambda)=\theta_1-(|\theta|-1)=1-\theta_2$. Therefore, \eqref{phi-11} holds because $\theta_2<1$.
    Similarly, \eqref{phi+10} holds because $2-\theta_2-(2\theta_1-a(\lambda)-b(\lambda))=1-\theta_1>0$, by symmetry \eqref{phi-01} holds because $\theta_2<1$, and \eqref{phi+00} holds because $\theta_1<1$.
\end{remark}
Provided that the series \eqref{H2F1} is well defined we have the following identity:
\begin{equation}\label{Fin2}
    \HG{a}{b}{c}{0} = 1.
\end{equation}
\subsection{Derivatives}
The $x$-derivative of a hypergeometric function is again a hypergeometric function given by the following formula:
\begin{equation}\label{hypergeoDerivative}
    \frac{\dif }{\dif x}\left(\HG{a}{b}{c}{x}\right)=\frac{ab}{c}\HG{a+1}{b+1}{c+1}{x}.
\end{equation}
When using this formula one needs to be careful since $c+1-(a+1+b+1)=c-a-b-1$, so that the act of differentiating reduces the value of the convergence-determining ``$c-a-b$'' compared to the original function. 
\begin{remark}\label{HypergeometricDifferentialEquationRemark}
    We use this formula for example when checking the boundary conditions in equation \eqref{condition0}, where the derivative is $\HG{a(\lambda)+1}{b(\lambda)+1}{\theta_1+1}{y}$. In that case we have $\theta_1+1-a(\lambda)-1-b(\lambda)-1=\theta_1-1-(\theta_1+\theta_2-1)=-\theta_2 \in (-1,0)$. Therefore we have the conditional convergence of the series  on $|y| < 1$.
\end{remark}
\subsection{Relation to the hypergeometric differential equation}
The hypergeometric differential equation is given by
\begin{align}\label{HypergeometricDifferentialEquation}
    x(1-x)\frac{\dif ^{2}f(x)}{\dif  x^{2}} + \left( c - (a+b+1)x \right)\frac{\dif  f(x)}{\dif  x} - ab f(x)=0
\end{align}
for some parameters $a,b,c$. Note first that the differential equation \eqref{HypergeometricDifferentialEquation} has three singularities: at 0, at 1 and at $\infty$. We need not worry about this last case as we are only interested in cases when $x\in[0,1]$. If none of the numbers $c,c-a-b,a-b$ is equal to an integer then the two linearly independent solutions around $x=0$ are increasing and given by
\begin{align}\label{H2F1@0}
    \HG{a}{b}{c}{x}, & & x^{1-c}\HG{a-c+1}{b-c+1}{2-c}{x},
\end{align}
whilst the solutions around $x=1$ are decreasing and given by
\begin{align}\label{H2F1@1}
    \HG{a}{b}{a+b+1-c}{1-x}, & & (1-x)^{c-a-b}\HG{c-a}{c-b}{c-a-b+1}{1-x}.
\end{align}
Figure \ref{fig:hypergeometrics} displays the behaviour of these functions over the interval $[0,1]$.
\begin{figure}[ht]
    \centering
    \includegraphics[scale=1]{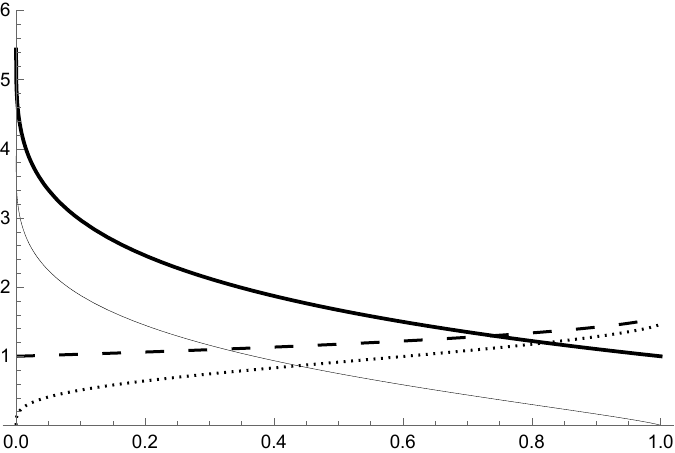}
    \caption{The four solutions to \eqref{HypergeometricDifferentialEquation}. Dashed curve is $\HG{a}{b}{c}{x}$, dotted is $x^{1-c}\HG{a-c+1}{b-c+1}{2-c}{x}$, bold is $\HG{a}{b}{a+b+1-c}{1-x}$, and thin is $(1-x)^{c-a-b}\HG{c-a}{c-b}{c-a-b+1}{1-x}$, all for $\theta_1=0.7$, $\theta_2=0.2$, and $\lambda=0.1$.}
    \label{fig:hypergeometrics}
\end{figure}

\begin{remark}\label{notalllambda}
    In our case it follows from \eqref{WFGenerator} and \eqref{GenEigFns} that $a=a(\lambda)$, $b=b(\lambda)$, and $c=\theta_1$. Therefore $c$ is not an integer and also $c-a-b=\theta_1-(\theta_1+\theta_2-1)=1-\theta_2$ is not an integer. To check that $a-b=\sqrt{(|\theta|-1)^2-8\lambda}$ is not an integer, note that $(|\theta|-1)^2 \in [0,1)$ and $\lambda\geq 0$. Therefore we only need to check if $a-b=0$, which happens only when $\lambda = (|\theta| - 1)^2 / 8$.
\end{remark}

Any three solutions are necessarily linearly related \cite[Theorem 2.3.2]{AndrewsAskeyRoy}; more specifically we have the following relations:
\begin{align}
    &\HG{a}{b}{a+b+1-c}{1-x}\notag\\
    &=\frac{\Gamma(1-c)\Gamma(a+b-c+1)}{\Gamma(a-c+1)\Gamma(b-c+1)} \HG{a}{b}{c}{x}\notag\\
    &\phantom{=}+\frac{\Gamma(c-1)\Gamma(a+b-c+1)}{\Gamma(a)\Gamma(b)} x^{1-c}\HG{a-c+1}{b-c+1}{2-c}{x},\label{lineardependence1}\\
    &(1-x)^{c-a-b}\HG{c-a}{c-b}{c-a-b+1}{1-x}\notag\\
    &=\frac{\Gamma(1-c)\Gamma(c-a-b+1)}{\Gamma(1-a)\Gamma(1-b)} \HG{a}{b}{c}{x}\notag\\
    &\phantom{=}+\frac{\Gamma(c-1)\Gamma(c-a-b+1)}{\Gamma(c-a)\Gamma(c-b)} x^{1-c}\HG{a-c+1}{b-c+1}{2-c}{x}.\label{lineardependence2}
\end{align}
Denote by $W(f_1,f_2)=f_1(x){f_2}'(x)-{f_1}'(x)f_2(x)$, and let $f$ and $g$ be the solutions \eqref{H2F1@0}, that is $f(x)=\HG{a}{b}{c}{x}$ and $g(x)=x^{1-c}\HG{a-c+1}{b-c+1}{2-c}{x}$.
Then we get that
\begin{equation}\label{wronskianHG}
    W(f,g)=(1-c)x^{-c}(1-x)^{c-a-b-1}.
\end{equation}
Additionally, if we denote by $h$ and $\kappa$ the solutions \eqref{H2F1@1}, that is
\begin{align*}
    h(x)&=\HG{a}{b}{a+b+1-c}{1-x}, \\
    \kappa(x)&=(1-x)^{c-a-b}\HG{c-a}{c-b}{c-a-b+1}{1-x},
\end{align*}
then we can use \eqref{lineardependence1} and \eqref{lineardependence2} together with the observation that if $W(f_1,f_2)=f_3$ and $f_4=\alpha f_1+\beta f_2$ for constant $\alpha, \beta$, then $W(f_1,f_4) = \beta f_3$, to deduce that
\begin{align}
    W(f,h)={}&\frac{\Gamma(c-1)\Gamma(a+b-c+1)}{\Gamma(a)\Gamma(b)}(1-c)x^{-c}(1-x)^{c-a-b-1},\label{W(Phi-,Phi+)}\\
    W(f,\kappa)={}&\frac{\Gamma(c-1)\Gamma(c-a-b+1)}{\Gamma(c-a)\Gamma(c-b)} (1-c)x^{-c}(1-x)^{c-a-b-1},\label{W(Phi1-,Phi1+)}\\
    W(g,h)={}&\frac{\Gamma(1-c)\Gamma(a+b-c+1)}{\Gamma(a-c+1)\Gamma(b-c+1)}(1-c)x^{-c}(1-x)^{c-a-b-1},\notag \\
    W(g,\kappa)={}&\frac{\Gamma(1-c)\Gamma(c-a-b+1)}{\Gamma(1-a)\Gamma(1-b)}(1-c)x^{-c}(1-x)^{c-a-b-1}. \notag
\end{align}

\section*{Acknowledgements}
Jere Koskela was supported by EPSRC research grant EP/V049208/1. This work was supported by The Alan Turing Institute under the EPSRC grant EP/N510129/1. Jaromir Sant was supported by the ERC (Starting Grant ARGPHENO 850869). Ivana Valenti{\'c} was supported by the Croatian Science Foundation under project IP-2022-10-2277. Victor Rivero joined this project while he was visiting the Department of Statistics at the University of Warwick, United Kingdom; he would like to thank his hosts for partial financial support as well as for their kindness and hospitality. In addition, VR is grateful for additional financial support  from  CONAHCyT-Mexico, grant nr.~852367. The authors are grateful for the many helpful comments of two anonymous referees, including one who provided formula \eqref{eq:greenproduct} and another who suggested our route to proving mutual singularity of path measures.

\bibliographystyle{abbrv}
\bibliography{bib}
%\nocite{*}

\end{document}